\documentclass[10pt,a4paper]{article}
\title{\bf\Large  On Perelman's $W$-entropy and Shannon entropy power for super Ricci flows on metric measure spaces\\
 }
\author{
Xiang-Dong Li\thanks{Research of X.-D. Li has been supported by National Key R\&D Program of China (No. 2020YF0712700) and NSFC No. 12171458.},  
	}
\date{}

\usepackage{amsfonts}
\usepackage{mathrsfs}

\usepackage{amssymb, latexsym}

\usepackage{graphicx}

\usepackage{amsmath}

\usepackage{amssymb,amsmath,amsfonts,amsthm,array,bm,color}

\def\Ric{{\rm Ric}}

\def\<{\langle}
\def\>{\rangle}

\newtheorem{theorem}{Theorem}[section]
\newtheorem{lemma}[theorem]{Lemma}       
\newtheorem{corollary}[theorem]{Corollary}

\newtheorem{remark}[theorem]{Remark}

\newtheorem{definition}[theorem]{Definition}

\usepackage{indentfirst}
\begin{document}

\maketitle
\makeatletter 
\renewcommand\theequation{\thesection.\arabic{equation}}
\@addtoreset{equation}{section}
\makeatother 

\vspace{-10mm}
\begin{center}
{\it Dedicated to  Jiahe for his 8th birthday}
\end{center}
\vspace{-2mm}
\begin{abstract} In this paper,  we extend Perelman's $W$-entropy formula and 
the concavity of the Shannon entropy power from smooth Ricci flow to super Ricci flows on metric measure spaces.  Moreover, we prove the Li-Yau-Hamilton-Perelman Harnack inequality on super Ricci flows. As a significant application, we prove 
the equivalence between the 
volume non-local collapsing property and the lower boundedness of the $W$-entropy on RCD$(0, N)$ spaces. 
Finally, we use the $W$-entropy to study the logarithmic Sobolev inequality with optimal constant 
on super Ricci flows on metric measure spaces. 
\end{abstract}


\vskip 2mm

\tableofcontents


\section{Introduction}\label{sect1}

In his seminal paper \cite{P1}, Perelman introduced the conjugate heat equation and the $W$-entropy on the closed  Ricci flow. More precisely, let $M$ be an $n$-dimensional closed manifold with a family of Riemannian metrics $\{g(t): t\in [0, T]\}$ and potentials $f\in C^\infty(M\times [0, T], \mathbb{R})$, where $T>0$. Suppose that
 $(g(t), f(t), \tau(t), t\in [0, T])$ is a solution to the evolution
equations
\begin{eqnarray}
\partial_t g=-2Ric, \ \ \partial_t f=-\Delta f+|\nabla
f|^2-R+{n\over 2\tau}, \ \ \partial_t \tau=-1, \label{r-c}
\end{eqnarray}
Following \cite{P1}, the $W$-entropy functional associated to $(\ref{r-c})$ is defined by
\begin{eqnarray}
W(g, f, \tau)=\int_M \left[\tau(R+|\nabla
f|^2)+f-n\right]{e^{-f}\over
 (4\pi\tau)^{n/2}}dv,\label{entropy-1}
 \end{eqnarray}
where $v$ denotes the volume measure on $(M, g)$. By \cite{P1}, the following
entropy formula holds
\begin{eqnarray}
{d\over dt}W(g, f, \tau)=2 \int_M \tau\left\|Ric+\nabla^2
f-{g\over 2\tau}\right\|_{\rm HS}^2{e^{-f}\over (4\pi \tau)^{n/2}}dv.\label{ent-p}
\end{eqnarray}
Here $\|\cdot\|_{\rm HS}$ denotes the Hilbert-Schmidt norm. 
In particular, the $W$-entropy is monotone increasing
in $t$ and the monotonicity is strict except that $M$ is a
shrinking Ricci soliton
\begin{eqnarray*}
Ric+\nabla^2 f={g\over 2\tau}.
\end{eqnarray*}
As an application of the $W$-entropy formula, Perelman \cite{P1} proved the non-local collapsing theorem 
for the Ricci flow  \cite{H2}, which plays an important r\^ole for ruling out cigars, one part of the singularity 
classification for the final resolution of the Poincar\'e conjecture and Thurston's geometrization conjecture 
on three dimensional closed manifolds. See also \cite{CZ, CLN, 
Kleiner-Lott, MT}.

Inspired by Perelman's groundbreaking contributions to the study of $W$-entropy formula, many authors have extended the $W$-entropy formula to various geometric flows. In \cite{N1, N2}, Ni derived the  $W$-entropy for the heat equation on  complete Riemannian manifolds. More precisely, let $(M, g)$ be a complete Riemannian manifold, and 

\begin{equation*}
	u(x, t)=\frac{e^{-f(x, t)}}{(4\pi t)^{\frac{n}{2}}}
\end{equation*}
be a positive solution to the  heat equation
\begin{equation*}
	\partial_t u=\Delta u \label{Heat1}
\end{equation*}
with $\int_M u(x,0)dv=1$. Define the $W$-entropy by
\begin{equation}\label{Wentopy}
	W(f, t):=\int_M \left( t|\nabla f |^2+f-n\right)\frac{e^{-f(x, t)}}{(4\pi t)^{\frac{n}{2}}} dv.
\end{equation}
Then the following $W$-entropy formula holds
\begin{equation*}
	\frac{d}{d t}W(f, t)=-2t\int_M \left(\Big\|\nabla^2 f-\frac{g}{2t} \Big\|_{\operatorname{HS}}^2
	+\Ric(\nabla f,\nabla f)\right)u dv. \label{WNi}
\end{equation*}
In particular,  the $W$-entropy is non-increasing when $\Ric\geq 0$. See Li-Xu \cite{LiXu} for the
extension of Ni's $W$-entropy formula for the heat equation $\partial_t u=\Delta u$ to complete Riemannian manifolds with $Ric\geq K$, $K\in \mathbb{R}$. 

In \cite{Li2012}, the author of this paper extended Perelman and Ni's $W$-entropy formulas to the heat equation of the Witten Laplacian on complete Riemannian manifolds with bounded geometry condition. 
More precisely, let $(M, g)$ be a complete Riemannian manifold with bounded geometric condition \footnote{Here, we say that $(M, g)$ satisfies the bounded geometry condition if the Riemannian curvature tensor $\mathrm{Riem}$ and its covariant derivatives $\nabla^k \mathrm{Riem}$ are uniformly bounded on $M$ for 
$k = 1, 2, 3$.}, $\phi\in C^4(M)$ with $\nabla \phi\in C_b^k(M)$ for $k=1, 2, 3$, let 
  \begin{equation*}
	u(x, t)=\frac{e^{-f(x, t)}}{(4\pi t)^{\frac{m}{2}}}
\end{equation*}
be a positive and smooth solution to the heat equation
\begin{equation}
	\partial_t u=Lu
	 \label{Heat2}
\end{equation}
with $\int_M u(x, 0)d\mu(x)=1`$. 
Define the $W$-entropy by
\begin{equation}\label{Wentopy}
	W_m(f, t):=\int_M \left( t|\nabla f |^2+f-m\right)\frac{e^{-f(x, t)}}{(4\pi t)^{\frac{m}{2}}} d\mu.
\end{equation}
Then the  following $W$-entropy formula holds
\begin{eqnarray}
	\frac{d}{dt}W_m(f, t)&=&-2t\int_M \left(\Big\|\nabla^2 f-\frac{g}{2t} \Big\|_{\operatorname{HS}}^2
	+\Ric_{m, n}(L)(\nabla f,\nabla f)\right)u d\mu \nonumber\\
	& &\hskip1cm -{2t\over m-n}\int_M\Big|\nabla \phi\cdot \nabla f-{m-n\over 2t}\Big|^2 u d\mu, \label{WLi}
\end{eqnarray}
In particular,  ${d\over dt}W_m(u(t))\leq 0$ (i.e., the $W$-entropy is non-increasing) on $[0, \infty)$ when $\Ric_{m, n}(L)\geq 0$. Moreover, under the assumption $Ric_{m, n}(L)\geq 0$, ${d\over dt}W_m(u(t))=0$  holds at some $t=t_0>0$ if and only if $(M, g)$ is isometric to the Euclidean space $\mathbb{R}^n$, $m=n$ and $V$ is identically a constant. This can be regarded as the rigidity theorem for the $W$-entropy on complete Riemannian manifolds 
 with $\Ric_{m, n}(L)\geq 0$. 

In \cite{LiLi2015PJM}, S. Li and the author of this paper gave  an alternative proof of $(\ref{WLi})$ using Ni's $W$-entropy formula $(\ref{WNi})$ and a warped product metric on $\mathcal{M}=M\times \mathbb{R}^{m-n}$ for when $m\in \mathbb{N}$ and $m>n$.   In a series of subsequently papers with  S. Li \cite{LiLi2015PJM, LiLi2018JFA, LiLi2018SCM, LiLi2020JGA}, we extended the $W$-entropy formula to the heat equation $\partial_t u=Lu$ associated with the 
Witten Laplacian on a  complete Riemannian manifolds satisfying the CD$(K, m)$-condition, 
i.e., $Ric_{m, n}(L)\geq K$, $K\in \mathbb{R}$ and $m\in [n, \infty]$. Moreover, we further extended 
Perelman's $W$-entropy formula to the heat equation $\partial_t u=Lu$ associated with the time-dependent 
Witten Laplacian on a  family of complete Riemannian manifolds $(M, g_t, \phi_t)$ with the so-called 
$(K, m)$-super Ricci flows. Very recently, S. Li and the author \cite{LiLi2015PJM} proved the $K$-concavity of the Shannon entropy power on complete Riemannian manifolds with $\Ric_{m, n}(L)\geq K$ and  on compact 
$(K, m)$-super flows.

It is natural to ask an interesting question whether one can  extend the monotonicity of $W$-entropy to 
 more singular spaces than smooth Riemannian manifolds. 
 In \cite{KL}. Kuwada and the  author of this paper proved 
 the monotonicity of $W$-entropy on the so-called RCD$(0, N)$ spaces and provided the associated rigidity results. For its precise definition and statement, see Section 3 and Section 4 below. As far as we know, this is the first result on the $W$-entropy and related topics on RCD spaces. Motivated by very  increasing interest of the study on the geometry and analysis on RCD spaces, it is natural and interesting to ask a question whether one can extend Kuwada-Li's and S.Li-Li's results to RCD$(K, N)$ spaces. This have been done in a recent paper by the author with his PhD student Enrui Zhang \cite{Li-Zhang2025}. See also an independent work by M. Brena \cite{Brena2025}.

The purpose of this paper is to study the $W$-entropy associated with the 
heat equation on the so-called $(K, N)$ or $(K, n, N)$-super Ricci flows on metric measure spaces. 
The main results of this paper extend the $W$-entropy formula and the monotonicity of the $W$-entropy from smooth Ricci flow, smooth $(K, m)$-super Ricci flows to the so-called $(K, N)$-super Ricci flow and $(K, n, N)$-super Ricci flows on metric measure spaces. To avoid the length of the Introduction part to be too long, we will introduce the notion of $(K, N)$-super Ricci flows and $(K, n, N)$-super Ricci flows on metric measure spaces and then state our main results in Section 4 below. 
 
The structure of this paper is as follows: 
In Section 2, for the convenience of the readers, we briefly review the $W$-entropy formulas on smooth $(K, m)$-super Ricci flows. In Section 3, we briefly recall some basic notions of RCD spaces and Sturm's  $(K, N)$ super Ricci flows on mm spaces, then we introduce the notion of $(K, n, N)$-super Ricci flows on mm spaces. In Section4, we first state the $H$-entropy dissipation formulas and then state main results of this paper. 
In Section 5, we prove the main theorems  of this paper. In Section 6, we prove the concavity of the Shannon entropy power and the related logarithmic entropy formula on closed super Ricci flows on mm spaces.  
In Section 7, we prove the Li-Yau-Hamilton-Perelman Harnack inequality on super Ricci flows.
As a significant application, we prove in Section 8 the equivalence between the 
volume non-local collapsing property and the lower boundedness of the $W$-entropy on 
RCD$(0, N)$ spaces. In Section 9, we use the $W$-entropy to study the logarithmic Sobolev inequality with optimal constant 
on super Ricci flows on metric measure spaces and raise a problem for the study in the future.

 In a forthcoming paper which is stilll in preparation, we will extend the $W$-entropy formula and Bakry-Ledoux's version of the L\'evy-Gromov isoperimetric inequality to the $(K, \infty)$-super Ricci flows on metric measure spaces.

\medskip

\noindent{\bf Acknowledgement}. The authors of this paper would like to express their gratitudes to Prof. Banxian Han, Dr. Songzi Li, Dr. Yuzhao Wang and M. Enrui Zhang for helpful discussions during the preparation of this paper. The author would like also to thank Prof. K.-T. Sturm for valuable discussion on super Ricci flows on metric measure spaces many years ago.

\section{$W$-entropy formulas on smooth super Ricci flows}

\subsection{Smooth $(K, m)$-super Ricci flows}

In the setting of smooth Riemannian manifolds, the notion of $(K, m)$-super Ricci flows has been independently introduced by S. Li and the author of this paper in our preprint (arxiv:1303.6019,  submitted on 25 Mar 2013) and 2015 published article \cite{LiLi2015PJM}. See also subsequent articles \cite{LiLi2018JFA, LiLi2018SCM, LiLi2020JGA}.

More precisely, let $(M, g_t, \phi_t, t\in [0, T])$ is a time-dependent weighted, $n$-dimensional Riemannian manifold $(X, g_t)$ with weighted volume measures $d\mu_t=e^{-\phi_t}dv_t$, and let the operator  $L_t$ be the time dependent Witten Laplacian on $(M, g_t, \phi_t)$ given by
$$L_t=\Delta_t-\nabla_t \phi_t\cdot\nabla_t,$$
where $dv_t=\sqrt{{\rm det} g_t(x)}dx$ is the standard Riemannian volume measure on $(M, g_t)$, $\nabla_t$ denotes the Levi-Civita covariant derivative on $(M, g_t)$ and $\Delta_t={\rm Tr}\nabla^2_t$ is the Laplace-Beltrami operator on $(M, g_t)$. Then $(M, g_t, \phi_t)$ is a $(K, m)$-super-Ricci flow for $N\geq n$ if and only if 
$$
{1\over 2}{\partial g_t\over \partial t}+ Ric_{m, n}(L_t)\geq Kg_t$$
where

$$
Ric_{m, n}(L_t):=Ric_{g_t}+\nabla^2 \phi_t-{\nabla \phi_t\otimes \nabla \phi_t\over m-n}$$
is the $m$-dimensional Bakry-Emery Ricci curvature associated with the time dependent Witten Laplacian $L_t$ on $(M, g_t, \phi_t)$. Note that,  $(M, g_t, \phi_t)$ is a super-$(K, m)$-Ricci flow for $m= n$ if and only if $\phi_t$ is constant in $x\in X$ for any fixed $t\in [0,. T]$.

In a series of papers \cite{LiLi2015PJM, LiLi2018JFA, LiLi2018SCM, LiLi2020JGA}, S. Li and the author of this paper extended the $W$-entropy formula to the heat equation $\partial_t u=Lu$ associated with the 
Witten Laplacian on a  complete Riemannian manifold satisfying the CD$(K, m)$-condition, 
i.e., $Ric_{m, n}(L)\geq K$, where $K\in \mathbb{R}$ and $m\in [n, \infty]$. Moreover, we further extended 
Perelman's $W$-entropy formula to the heat equation $\partial_t u=Lu$ associated with the time-dependent 
Witten Laplacian on a  family of complete Riemannian manifolds $(M, g_t, \phi_t)$ with the so-called 
$(K, m)$-super Ricci flow. 
In this section, for the convenience of the readers, we briefly review  the $W$-entropy formulas for the heat equation of the time-dependent Witten Laplacian on compact or complete $(K, m)$-super Ricci flows.

\subsection{$W$-entropy for  $(0, m)$-super Ricci flow}

In \cite{LiLi2015PJM}, S. Li and the author of this paper proved the $W$-entropy formula to the heat equation associated with  the time dependent Witten Laplacian on compact manifolds equipped with a $(0, m)$-super Ricci flow, which 
can be regarded as the $m$-dimensional analogue of Perelman's $W$-entropy formula for the Ricci flow.

\begin{theorem}\label{ThB} (\cite{LiLi2015PJM})  Let $(M, g(t), \phi(t), t\in [0, T])$ be a compact manifold with family of time dependent metrics and $C^2$-potentials. Suppose that $g(t)$ and $\phi(t)$
satisfy the conjugate equation 
\begin{eqnarray}\label{conjugate}
\frac{\partial \phi_t} {\partial t}={1\over 2}{\rm Tr}\left(
\frac{\partial g}{\partial t}\right).
\end{eqnarray}
Let  $u={e^{-f}\over (4\pi t)^{m/2}}$ be a positive and smooth solution of the heat equation
\begin{eqnarray*}
\partial_t u = L_tu
\end{eqnarray*}
with initial data $u(0)$ satisfying $\int_M
u(0)d\mu(0)=1$.
Let
\begin{eqnarray*}
H_m(u, t)=-\int_M u\log u d\mu-{m\over 2}(1+\log(4\pi t)).
\end{eqnarray*}
Define
\begin{eqnarray*}
W_m(u, t)={d\over dt}(tH_m(u)).
\end{eqnarray*}
Then
\begin{eqnarray*}
W_m(u, t)=\int_M \left[t|\nabla f|^2+f-m\right]ud\mu,
\end{eqnarray*}
and
\begin{eqnarray}
& &{d\over dt}W_m(u, t)=-2t\int_M \left\|\nabla^2 f-{g\over 2t}\right\|_{\rm HS}^2ud\mu-{2t\over m-n}\int_M \left(\nabla \phi\cdot \nabla f+{m-n\over 2t}\right)^2  ud\mu\nonumber\\
& &\hskip4cm -2t\int_M \left({1\over 2}{\partial g\over \partial t}+Ric_{m, n}(L)\right)(\nabla f, \nabla f)ud\mu.\label{NW}
\end{eqnarray}
In particular, if $\{g(t), \phi(t), t\in (0, T]\}$ is a $(0, m)$-super Ricci flow and satisfies the conjugate equation $(\ref{conjugate})$, then $W_m(u, t)$ is decreasing in $t\in (0, T]$, i.e.,
\begin{eqnarray*}
{d\over dt}W_m(u, t)\leq 0, \ \ \ \forall t\in (0, T].
\end{eqnarray*}
Moreover, the left hand side in $(\ref{NW})$ identically  equals to zero  on $(0, T]$ if and only if $(M, g(t), \phi(t), t\in (0, T])$ is a $(0, m)$-Ricci flow in the sense that
\begin{eqnarray*}
{\partial g\over \partial t}&=&-2{\rm Ric}_{m,n}(L),\\
{\partial \phi\over \partial t}&=&{1\over 2} {\rm Tr}\left( {\partial g\over \partial t}\right).
\end{eqnarray*}

\end{theorem}

\subsection{$W$-entropy for $(K, m)$-super Ricci flow}

In general we have the following result which extends Theorem \ref{ThB} to $(K, m)$-super Ricci flow for general $K\in \mathbb{R}$ and $m\in [n, \infty)$.

\begin{theorem}\label{Th-E} (\cite{LiLi2015PJM, LiLi2018JFA, LiLi2018SCM}) Let $(M, g(t), \phi(t), t\in [0, T])$ be a compact manifold with family of time dependent metrics and $C^2$-potentials. Suppose that $g(t)$ and $\phi(t)$
satisfy the conjugate equation $(\ref{conjugate})$. Let $u$ be a positive and smooth solution to the heat
equation $\partial_t u=L_tu$. Define 
\begin{eqnarray}
H_{m,K}(u, t)=-\int_M u\log u d\mu-{m\over 2}(1+\log(4\pi t))-\frac m2Kt\Big(1+\frac16Kt\Big),  \label{HmK}
\end{eqnarray}
and
\begin{align}
W_{m,K}(u, t)={d\over dt}(tH_{m,K}(u)). \label{WmK}
\end{align}
Then
\begin{eqnarray*}
W_{m,K}(u, t)=\int_M\left[t|\nabla f |^2+f-m\Big(1+\frac12Kt\Big)^2\right]ud\mu,\label{WmK-0}
\end{eqnarray*}
and
\begin{align}\label{WMK}
{d\over dt}W_{m,K}(u, t)&=-2 t\int_M\Big\|\nabla^2 f-\left(\frac1{2t}+\frac{K}{2}\right)g\Big\|_{\rm HS}^2u d\mu\nonumber\\
&\hskip1cm -\frac{2t}{m-n}\int_M\left(\nabla \phi\cdot\nabla f+(m-n)\Big(\frac1{2t}+\frac K2\Big)\right)^2u d\mu\nonumber\\
&\hskip1.5cm -2 t\int_M\left({1\over 2}{\partial g\over \partial t}+{\rm Ric}_{m,n}(L)+Kg\right)(\nabla f, \nabla f) ud\mu.
\end{align}
In particular, if $(M, g(t), \phi(t), t\in (0, T])$ is a $(-K, m)$-super Ricci flow and satisfies the conjugate equation $(\ref{conjugate})$, then $W_{m,K}(u, t)$ is decreasing in $t\in (0, T]$, i.e.,
\begin{eqnarray*}
{d\over dt}W_{m,K}(u, t)\leq 0, \ \ \ \forall t\in (0, T].
\end{eqnarray*}
Moreover, the left hand side in $(\ref{WMK})$ identically  equals to zero on $(0, T]$ if and only if $(M, g(t), \phi(t), t\in (0, T])$ is a $(-K, m)$-Ricci flow in the sense that
\begin{eqnarray*}
{\partial g\over \partial t}&=&-2({\rm Ric}_{m,n}(L)+Kg),\\
{\partial \phi\over \partial t}&=&{1\over 2} {\rm Tr}\left( {\partial g\over \partial t}\right).
\end{eqnarray*}
\end{theorem}

For the $W$-entropy formula   for the time dependent
Witten Laplacian on compact Riemannian manifolds with $(K, \infty)$-super Ricci flow, see S. Li-Li \cite{LiLi2018SCM, LiLi2020JGA}. 

\section{Super Ricci flows on metric measure  spaces}\label{sect2}

\subsection{Basic facts about RCD spaces}

According to \cite{Gigili2015, Gigili2015MAMS, BG}, an RCD$(K , N )$ space is an infinitesimally Hilbertian 
metric measure space $(X, d, m)$ satisfying a lower Ricci curvature bound and an upper dimension bound (meaningful if $N<\infty$) in a synthetic sense according to \cite{LV2009, Sturm2006}. For the convenience of the readers, we briefly recall its definition and basic facts.

 Let $(X, d, \mu)$ be a metric measure space, which means that $(X, d)$ is 
a complete and separable metric space and $\mu$ is a locally finite measure.
Locally finite means that for all $x \in X$ , there is $r>0$ such that 
$\mu\left(B_r(x)\right)<\infty$ and $\mu$ is a  $\sigma$-finite Borel measure on $X$, where $B_r(x)=\{y\in X, d(x, y)<r\}$.

 Let  $P_2(X, d)$ be the $L^2$-Wasserstein space over $(X, d)$, i.e. the set of all Borel probability measures $\mu$ satisfying
$$
\int_X d\left(x_0, x\right)^2 \mu(\mathrm{d} x)<\infty,
$$
where $x_0 \in X$ is a (and hence any) fixed point in $M$. The $L^2$-Wasserstein distance between $\mu_0, \mu_1 \in P_2(X, d)$ is defined by
$$
W_2\left(\mu_0, \mu_1\right)^2:=\inf_{\pi\in \Pi} \int_{X\times X} d(x, y)^2 \mathrm{~d} \pi(x, y),
$$
where $\Pi$ is the set of coupling measures $\pi$ of $\mu_0$ and $ \mu_1$ on $X\times X$, i.e., $\Pi=\{\pi\in P(X\times X), \pi(\cdot, X)=\mu_0, \pi(X, \cdot)=\mu_1\}$, where $P(X\times X)$ is the set of probability measures on $X\times X$.

Fix a reference measure $\mu$ on $(X, d)$, let  $P_2(X, d, \mu)$ be the subspace of  all absolutely continuous measures  with respect to the measure $\mu$.  For any given measure $\nu \in P_2(X, d)$, we can define the relative entropy with respect to $\mu$ as 
$$
\operatorname{Ent}(\nu):=\int_X \rho \log \rho \mathrm{d} \mu,
$$
if $\nu=\rho \mu$ is absolutely continuous w.r.t.\;$\mu$ and $(\rho \log \rho)_{+}$ is integrable  w.r.t.\;$\mu$,  otherwise we set $\operatorname{Ent}(\nu)=+\infty$.
 The Fisher information is defined by
$$
I(\nu):= \begin{cases}\int_X \frac{|D \rho|^2}{\rho} \mathrm{d}\mu & \text { if } \nu=\rho \mu, \\ \infty & \text { otherwise. }\end{cases}
$$
 
Given $N \in$ $(0, \infty)$, Ebar, Kuwada and Sturm \cite{EKS2015} introduced the functional $U_N: P_2(X, d) \rightarrow[0, \infty]$ 

$$
U_N(\nu):=\exp \left(-\frac{1}{N} \operatorname{Ent}(\nu)\right),
$$
which is similar to the Shannon entropy power \cite{Sh}.

 We now follow Bacher  and Sturm \cite{BS2010} and Ambrosio-Gigli-Savar\'e   \cite{AGS2014Invent} to introduce 
  the definition of CD$^*(K, N)$ and RCD$^*(K, N)$ spaces below. Let $P_{\infty}(X, d, \mu)$ be the set of measures in $P_2(X, d, \mu)$ with bounded support.
 
\begin{definition}\cite{EKS2015}\label{Def}
 	For $\kappa \in \mathbb{R}$, and $\theta \geq 0 $ we define the function
 	\begin{equation*}
\mathfrak{s}_\kappa(\theta)= \begin{cases}\frac{1}{\sqrt{\kappa}} \sin (\sqrt{\kappa} \theta), & \kappa>0 ,\\ \theta, & \kappa=0 ,\\ \frac{1}{\sqrt{-\kappa}} \sinh (\sqrt{-\kappa} \theta), & \kappa<0.\end{cases}
\end{equation*}
\begin{equation*}
\mathfrak{c}_\kappa(\theta)= \begin{cases}\cos (\sqrt{\kappa} \theta), & \kappa \geq 0 ,\\ \cosh (\sqrt{-\kappa} \theta), & \kappa<0.\end{cases}
\end{equation*}
Moreover, for $t\in [0,1]$ we set
\begin{equation*}
\sigma_\kappa^{(t)}(\theta)= \begin{cases}\frac{\mathfrak{s}_\kappa(t \theta)}{\mathfrak{s}_\kappa(\theta)}, & \kappa \theta^2 \neq 0 \text { and } \kappa \theta^2<\pi^2 ,\\ t, & \kappa \theta^2=0 ,\\ +\infty, & \kappa \theta^2 \geq \pi^2.\end{cases}
\end{equation*}
 \end{definition} 
 
\begin{definition}[\cite{BS2010}]
	   We say that metric measure space $(X,d,\mu)$ satisfies the reduced curvature-dimension condition CD$^\ast(K, N)$ if and only if for each pair $\mu_0=\rho_0 \mu, \mu_1=\rho_1 \mu \in P_{\infty}(X, d, \mu)$, there exists an optimal coupling $\pi$ of $\mu_0$ and $\mu_1$ such that
\begin{equation}\label{CD}
\begin{aligned}
\int_X \rho_t^{-\frac{1}{N^{\prime}}} \mathrm{d} \mu_t \geq & \int_{X \times X}\left[\sigma_{K / N^{\prime}}^{(1-t)}\left(d\left(x_0, x_1\right)\right) \rho_0\left(x_0\right)^{-\frac{1}{N^{\prime}}}\right. \\
& \left.+\sigma_{K / N^{\prime}}^{(t)}\left(d\left(x_0, x_1\right)\right) \rho_1\left(x_1\right)^{-\frac{1}{N^{\prime}}}\right] \mathrm{d} \pi\left(x_0, x_1\right),
\end{aligned}
\end{equation}
 where  $\left(\mu_t\right)_{t \in[0,1]}$ in $P_{\infty}(X, d, \mu)$ is a geodesic connecting $\mu_0$ and $\mu_1$ and $N^{\prime} \geq N$.
If inequality \eqref{CD} holds for any geodesic $\left(\mu_t\right)_{t \in[0,1]}$ in $P_{\infty}(X, d, \mu)$,  we say that $(X, d, \mu)$ is a strong CD$^*(K, N)$ space.
\end{definition}

  To introduce the RCD spaces and  consider the canonical heat flow on $(X, d, \mu)$, we need several  notions  including the Cheeger energy functional.

\begin{definition}[minimal relaxed gradient\cite{AGS2014Invent}]
		We say that $G \in L^2(X, \mu)$ is a relaxed gradient of $f \in L^2(X, \mu)$ if there exist Borel d-Lipschitz functions $f_n \in L^2(X, \mu)$ such that: 
		
		(a) $f_n \rightarrow f$ in $L^2(X, \mu)$ and $|Df_n|$ weakly converge to  $\tilde{G}$ in $L^2(X, \mu)$; 
		
		(b) $\tilde{G} \leq G $. m-a.e. in X.  We say that G is the minimal relaxed gradient of f if its $L^2(X, \mu)$ norm is minimal among relaxed gradients. 
		
 We use $|Df |_{\ast} $ to denote  the minimal relaxed gradient.
\end{definition}
%
 Ambrosio et al \cite{AGS2014} proved that $|Df |_{\ast}=|\nabla f |_w , \mu-a.e$ where  $|\nabla f |_w $ denotes the so called minimal weak upper
gradient of $f$ (cf \cite{AGS2014Invent})

The Cheeger energy functional \cite{EKS2015} is  defined by  
\begin{equation*}
	\operatorname{Ch}(f):=\int_X |\nabla f |^2_w d\mu,
\end{equation*}
  and inner product is given by $$
\langle\nabla f, \nabla g\rangle:=\lim _{\varepsilon \searrow 0} \frac{1}{2 \varepsilon}\left(|\nabla(f+\varepsilon g)|_w^2-|\nabla f|_w^2\right).
$$
  We now have a strongly local Dirichlet form $(\mathcal{E}, D(\mathcal{E}))$ on $L^2(X, \mu)$ by setting $\mathcal{E}(f, f)=\operatorname{Ch}(f)$ and $D(\mathcal{E})=W^{1,2}(X, d, \mu)$ being a Hilbert space 
  and $L^2$-Lipschitz 
functions are dense in the usual sense.  In this case, $\mathrm{H}_t$ is a semigroup of the self-adjoined linear operator on $L^2(X, \mu)$ with the Laplacian $\Delta$ as its generator. The previous result implies that for $f, g \in W^{1,2}(X, d, \mu)$,  the Dirichlet form is defined by
$$
\mathcal{E}(f, g):=\int_X\langle\nabla f, \nabla g\rangle \mathrm{d} \mu.
$$
Moreover, for $f \in W^{1,2}$ and $g \in D(\Delta)$, the integration by parts formula holds
$$
\int_X\langle\nabla f, \nabla g\rangle \mathrm{d} \mu=-\int_X f \Delta g \mathrm{~d} \mu.
$$

Ambrosio et al \cite{AGS2014} proved that the Cheeger energy ${\rm Ch}$ is quadratic
 is equivalent to the linearity of the heat semigroup $\mathrm{H}_t$ defined by solving the 
 heat equation below:
\begin{equation*}
	\frac{\partial}{\partial t}u=\Delta u, \ \ \  u(0)=f.
\end{equation*}

For any $f, g\in D(\Delta)\cap W^{1, 2}(X, \mu)$ with $\Delta f, \Delta g\in W^{1, 2}(X, \mu)$, the iterated carr\'e du champs {\bf measure} is defined  by 
		\begin{eqnarray*}
	{\bf \Gamma_2}(f,  g):={1\over 2}\Delta \langle \nabla f, \nabla g\rangle-{1\over 2}\left(\langle\nabla f, \nabla \Delta g\rangle+\langle\nabla g, \nabla \Delta f\rangle\right)\mu.
	\end{eqnarray*}
	In particular,  we have
	\begin{eqnarray*}
	{\bf \Gamma_2}(f, f):={1\over 2}\Delta |\nabla f|^2-\langle\nabla f, \nabla \Delta f\rangle 
	\mu.
	\end{eqnarray*}

\begin{definition}[\cite{EKS2015, KL}]
	We say that a metric measure space $(X, d, \mu)$ is infinitesimally Hilbertian if the associated
	 Cheeger energy is quadratic. Moreover,
we call $(X, d, \mu)$ an RCD$^*(K, N)$ space if it satisfies the reduced Curvature-Dimension condition CD$^*(K, N)$ and satisfies infinitesimally Hilbertian condition. 	
\end{definition}

By \cite{AGS2014,EKS2015,KL},  for infinitesimally Hilbertian $(X, d, \mu)$, CD$^*(K, N)$  condition is equivalent to the following  conditions:

(i) There exists $C>0$ and $x_0 \in X$ such that
\begin{equation*}
\int_X \mathrm{e}^{-C d\left(x_0, x\right)^2} \mu(\mathrm{~d} x)<\infty.
\end{equation*}

(ii) For $f \in \mathcal{D}(\mathrm{Ch})$ with $|\nabla f|_w \leq 1 \quad \mu$-a.e.  $f$ has a $1$-Lipschitz representative.

(iii) For all $f \in \mathcal{D}(\Delta)$ with $\Delta f \in \mathcal{D}(\Delta)$ and $g \in \mathcal{D}(\Delta) \cap L^{\infty}(\mu)$ with $g \geq 0$ and $\Delta g \in L^{\infty}(\mu)$
$$
\begin{aligned}
& \frac{1}{2} \int_X|\nabla f|_w^2 \Delta g \mathrm{~d} \mu-\int_X\langle \nabla f, \nabla \Delta f\rangle g \mathrm{d}\mu 
 \geq K \int_X|\nabla f|_w^2 g \mathrm{d}\mu+\frac{1}{N} \int_X(\Delta f)^2 g \mathrm{d}\mu.
\end{aligned}
$$

We now give the following important examples of RCD$^\ast(K,N)$ space:
\medskip

\begin{itemize}

\item Let $\left(M^n, g\right)$ be a complete Riemannian manifold, $f: M \rightarrow \mathbb{R}$ a $C^2(M)$ function, $d_g$ the Riemannian distance function, and vol $g_g$ the Riemannian volume measure on $M$. Set $\mathfrak{m}:=e^{-f}$ vol $_g$. Then the metric measure space ( $M, d_g, \mathfrak{m}$ ) satisfies $\operatorname{RCD}(K, N)$ condition for $N>n$ if and only if 
$$
\operatorname{Ric}_N:=\operatorname{Ric}_g+\operatorname{Hess}_f-\frac{d f \otimes d f}{N-n} \geq K g
$$
holds. For $N=n$, the $\operatorname{RCD}(K, n)$ condition is equivalent to $d f=0$ and $\operatorname{Ric}_g \geq K$.

\item Let $\left\{\left(X_i, d_i, \mathfrak{m}_i\right)\right\}_i$ be a family of $\operatorname{RCD}^*\left(K_i, N\right)$ spaces. For $x_i \in X_i$, assume $\mathfrak{m}_i\left(B_1\left(x_i\right)\right)=1, K_i \rightarrow K$ and $\left(X_i, d_i, \mathfrak{m}_i, x_i\right) \xrightarrow{p m G}\left(X_{\infty}, d_{\infty}, \mathfrak{m}_{\infty}, x_{\infty}\right)$ as $i \rightarrow \infty$, where $\xrightarrow{p m G}$ means the pointed measured Gromov convergence (see \cite{Gigili2015} ). Then $\left(X_{\infty}, d_{\infty}, \mathfrak{m}_{\infty}\right)$ satisfies the $\operatorname{RCD}^*(K, N)$ condition. Moreover a family of $\mathrm{RCD}^*(K, N)$ spaces with the normalized measures is precompact with respect to the pmG-convergence.

\end{itemize}

By  Cavaletti-Milman \cite{CavMim} and Z. Li \cite{LZH}, the notion of RCD$^*(K, N)$ space is indeed equivalent to the one of RCD$(K, N)$ space.  So we will only say RCD$(K, N)$ space throughout this paper.

  We now explain some basic results on RCD spaces. 
  For $f, g \in \mathcal{D}(\Delta)\cap \mathcal{L}^{\infty}(\mu)$ and $\varphi \in C^1(\mathbb{R})$ with $\varphi(0)=0$, we have $\varphi(f) \in \mathcal{D}(\Delta) \cap L^{\infty}(\mu)$ and  the following chain rule \eqref{nabla} (see \cite{FOT} ) and the Leibniz rule \eqref{Delta} for the Laplacian (see \cite{Gigili2015MAMS} ) hold :
\begin{equation}\label{nabla}
\begin{aligned}
	\langle \nabla \varphi(f), \nabla g\rangle &=\varphi^{\prime}(f)\langle \nabla f, \nabla g\rangle \quad \mu \text {-a.e. }\\
	\Delta(\phi(g) )& =\phi'(g) \Delta g +\phi''(g) | \nabla g |_w\quad \mu \text {-a.e. }
\end{aligned}
\end{equation}
\begin{equation}\label{Delta}
	\Delta (f\cdot g)
= f \Delta g +g \Delta f +2 <\nabla f, \nabla g> 
\end{equation}

 Ambrosio et al. \cite{AGS2014Invent} proved that for $\mu \in \mathcal{P}_2(X), t \mapsto \operatorname{Ent}\left(P_t \mu\right)$ is absolutely continuous on $(0, \infty)$ and $\mu_t=P_t \mu$ satisfies the energy dissipation identity, i.e. $\mu_t \rightarrow \mu_0$ as $t \rightarrow 0$ and for $0<s<t$,
\begin{equation}\label{EnDI}
\operatorname{Ent}\left(\mu_s\right)=\operatorname{Ent}\left(\mu_t\right)+\frac{1}{2} \int_s^t\left|\dot{\mu}_r\right|^2 \mathrm{~d} r+\frac{1}{2} \int_s^t I\left(\mu_r\right) \mathrm{d} r  \text { a.e. } t.
\end{equation}
 The energy dissipation identity \eqref{EnDI} is equivalent to the following equality
\begin{equation*}
-\frac{\mathrm{d}}{\mathrm{~d} t} \operatorname{Ent}\left(\mu_t\right)=\left|\dot{\mu}_t\right|^2=I\left(\mu_t\right)<\infty \quad \text { a.e.}\quad t. 
\end{equation*}

\subsection{Sturm's super Ricci flows on metric measure spaces }

In this subsection, we briefly follow \cite{Sturm18, Ko-Sturm18} to recall the notion of Sturm's 
super Ricci flows on metric measure spaces.

Let $(L_t)_{t\in (0, T)}$ be a $1$-parameter family
of linear operators defined on an algebra $\mathcal{A}$ of functions on a set $X$ such that 
$L_t(\mathcal{A})\subset \mathcal{A}$ for each $t\in [0, T]$. We assume that we are given a topology on $\mathcal{A}$ 
such that limits and derivatives make sense. In terms of these data we define the square field 
operators 
$$\Gamma_t(f, g) ={1\over 2} [L_t(fg)-L_tfg-fL_tg].$$ 

We assume that $L_t$ is a diffusion operator in the sense that

\begin{itemize}

\item $\Gamma_t(u, u)\geq 0$ for all $u\in \mathcal{A}$, 

\item for every $k$-tuple of functions $u_1, ..., u_k$ in $\mathcal{A}$ 
and every $C^\infty$-function $\psi: \mathbb{R}^k\rightarrow \mathbb{R}$ 
vanishing at the origin, $\psi(u_1, \ldots, u_k)\in \mathcal{A}$  and 
$$L_t \psi (u_1, \ldots, u_k)==\sum\limits_{i=1}^k\psi_i(u_1, \ldots, u_k)L_tu_i+
\sum\limits_{1\leq i, j\leq k}\psi_{ij}(u_1, \ldots, u_k) \Gamma_t(u_i, u_j),$$
where $\psi_i:={\partial \psi\over \partial y_i}$ and $\psi_{ij}:={\partial^2 \psi\over \partial y_i\partial y_j}$.

\end{itemize}

The Hessian of $u$ at time $t$ and a point $x\in X$ is the bilinear form on $\mathcal{A}$ 
$$
H_t u(v, w)(x)={1\over 2}\left[\Gamma_t(v, \Gamma_t(u, w))
+\Gamma_t(w, \Gamma_t(u, v))-\Gamma_t(u, \Gamma_t(v, w))
\right](x)
$$
for $u, v, w\in \mathcal{A}$. 

The $\Gamma_2$-operator is defined via iteration of the 
square field operator as

$$
\Gamma_{2, t} (u, v)(x)={1\over 2}\left[L_t \Gamma_t(u, v)-
+\Gamma_t(L_t u, v)-\Gamma_t(u, L_t v)\right](x).
$$
Too simplify the notation, let $\Gamma_t(u) =\Gamma_t(u, u)$ and $\Gamma_{2, t}(u) 
=\Gamma_{2, t}(u, u)$. 

In terms of the $\Gamma_2$-operator we 
define the Ricci tensor at the space-time point $(t, x)\in [0, T) 
\times X$ by
$$R_t(x)=\inf \{\Gamma_{2, t}(u+v)(x): v\in \mathcal{A}_x^0\}$$
for $u\in \mathcal{A}$ where 

$$\mathcal{A}_x^0=\{v=\psi(v_1, \ldots, v_k): k\in \mathbb{N}, v_1, \ldots, v_k\in \mathcal{A}, \psi \ {\rm smooth\ \ with}\ \psi_i(v_1, \ldots, v_k)(x)=0\ \forall i\}
$$

We can always extend the definition of $L_t$ and $\Gamma_t$ to the algebra generated by the elements in $\mathcal{A}$ and the constant functions which leads to $L_t1 =0$ and $\Gamma_t(1 , f) =0$ for all $f$.

For the sequel we assume in addition that we are given a $2$-parameter family $(P_t^s, 0\leq s\leq t\leq T)$ of linear operators on $\mathcal{A}$ satisfying for all $s \leq r\leq t$ and all $u\in \mathcal{A}$
\begin{equation}
\begin{aligned}
& P_t^t u = u, \ \ P_t^r (P_r^s u)=P_t^s u,\\
& (P_t^s u)^2 \leq P_t^s(u^2),\\
&s\rightarrow P_t^s u \ \ {\rm and}\ \  t \rightarrow P_t^s u\ \ {\rm continuous}\\
&
\partial_s P_t^s u=-P_t^s L_su\\
&
\partial_t P_t^s u=L_t P_t^s u.
\end{aligned}
\end{equation}
Such a propagator $(P_t^s)$ for the given family of operators $(L_t)$ will exist in quite general situations under mild assumptions. We also require that for each $1$-parameter family $(u_r)_{r\in (s, t)}$ which is differentiable within $\mathcal{A}$ w.r.t. $r$
$$
\partial_r P_t^s u_r  = P_t^s (\partial_r u_r), \ \ \ 
\partial_r \Gamma_t( u_r, v)  = \Gamma_t (\partial_r u_r, v).$$

\medskip


\noindent{\bf Definition B.1 in \cite{Sturm18}}. We say that $(L_t)_{t\in [0,T)}$ is a {\bf super-Ricci flow} if
$$\partial_t \Gamma_t\leq 2R_t.$$
It is called Ricci flow if 
$$\partial_t \Gamma_t= 2R_t.$$

\noindent{\bf Lemma B.2 in \cite{Sturm18}}. $(L_t)_{t\in [0,T)}$ is a super-Ricci flow if and only if
$$\partial_t \Gamma_t\leq 2\Gamma_{2, t}.$$

\noindent{\bf Lemma B.3 in \cite{Sturm18}}.  $(L_t)_{t\in [0,T)}$ is a super-Ricci flow if and only  in addition to $(82)$ for each $x$, each $\varepsilon>0$ and each $u \in \mathcal{A}$ there exists $v\in \mathcal{A}_0^x$ such that
$$\partial_t \Gamma_t(u)(x)+\varepsilon\leq 2\Gamma_{2, t}(u+v)(x).$$

\medskip

Given any extended number $N\in [1, \infty]$ we define the $N$-Ricci tensor at $(t, x)$ by
$$R_{N, t}(x)=\inf \{\Gamma_{2, t}(u+v)(x)-{1\over N}(L_t(u+v))^2(x): v\in \mathcal{A}_x^0\}$$
(Again recall that the definition of RN here slightly differs from that in [46].)

\medskip

\noindent{\bf Definition B.7 in \cite{Sturm18}}. We say that  $(L_t)_{t\in [0,T)}$ is a {\bf super-$N$-Ricci flow} if 
$$\partial_t \Gamma_t\leq 2_{N, t}.$$
If equality holds then $(L_t)_{t\in [0,T)}$ is $N$-Ricci flow.

A Ricci flow is a $N$-Ricci flow for the particular choice $N=\infty$, i.e. a solution to $\partial_t \Gamma_t=2R_t$.

\noindent{\bf Theorem B.8 in \cite{Sturm18}}. Under appropriate regularity assumptions on $(P_{t}^{s})_{s\leq t}$, the following are equivalent\\
$(i)$ $\partial_t \Gamma_t\leq 2R_{N, t}$ ($\forall u\in \mathcal{A}, \forall t$)\\
\\
$(ii)$ $\partial_t \Gamma_t(u)\leq 2\Gamma_{2, t}(u)-{2\over N}(L_t u)^2$ ($\forall u\in \mathcal{A}, \forall t$)\\
\\
$(iii)$ $\Gamma_t(P_t^s u) +2N\int_s^t(P_t^r L_rP_r^su)^2dr\leq P_t^s(\Gamma_s(u))$   ($\forall u\in \mathcal{A}, \forall s\leq t$)

\medskip

\medskip

 In \cite{Ko-Sturm18}, Kopfer and Sturm proved the following equivalences between  the super Ricci flows and the dynamic $(K, N)$-convexity of the Boltzmann entropy $S_t; [0, T]\times \mathcal{P}(X)\rightarrow (-\infty, +\infty]$
defined by
 
$$S_t(\mu)=\int_X u \log u dm_t \ \  {\rm if}\ \ \mu=um_t$$
and $S_t(\mu_)=+\infty$ if $\mu$ is not absolutely continuous with respect to $m_t$.

\noindent{\bf Theorem 1.9 in \cite{Ko-Sturm18}}.  For each $N\in (0, \infty)$ the following are equivalent:\\
$({\bf I}_N)$ For a.e. $t\in (0, T)$and every $W_t$ -geodesic $\mu^{a}, a\in [0, 1]$ in $\mathcal{P}$ 
 with $\mu^0, \mu^1\in {\rm Dom}(S)$
 \begin{equation}
\left. \partial_a^{+}S_t(\mu^a)\right|_{a=1-}-
\left. \partial_a^{-}S_t(\mu^a)\right|_{a=0+}
\geq -{1\over 2} \partial_t^{-}
W_{t-}^2(\mu^0, \mu^1)+{1\over N}|S_t(\mu^0)-S_t(\mu^1)|^2.
 \end{equation}
 $({\bf II}_N)$ For all $0\leq s\leq t\leq T$ and $\mu, \nu
 \in \mathcal{P}$
\begin{equation}
W^2_s(\widehat{P}_{t, s}\mu, \widehat{P}_{t, s}\nu)
\leq W^2_{t}(\mu, \nu)-{2\over N}\int_s^t [S_r( \widehat{P}_{t, r}\mu)-
S_r( \widehat{P}_{t, r}\nu)]^2dr,
 \end{equation}
  where $t\mapsto \mu_t=\widehat{P}_{\tau, t}$ is the dual heat flow which is unique dynamical backward EVI$^{-}$-gradient flow for the Boltzmann entropy $S$ in the following sense: for every $\mu\in {\rm Dom}(S)$ and every $\tau<T$ the absolutely continuous curve $t\mapsto \mu_t$ satisfies 
  $${1\over 2}\partial_s^{-}\left.W_{s, t}^2(\mu_s, \sigma)\right|_{s=t-}\geq S_t(\mu_t)-S_t(\sigma)$$
  for all $\sigma\in {\rm Dom}(S)$ and all $t\leq \tau.$\\  
$({\bf III}_N)$ For all $u\in {\rm Dom}(E)$
and  $0\leq s\leq t\leq T$  
\begin{equation}
|\nabla_t (P_{t, s}u)|^2\leq P_{t, s}(|\nabla_s u|^2)-{2\over N}
\int_s^t (P_{t, r}\Delta_r P_{r, s}u)^2 dr.
 \end{equation}
$({\bf IV}_N)$ For all $0\leq s\leq t\leq T$ and 
for all $u_s, g_t\in \mathcal{F}$  with $g_0\geq 0, 
g_t\in L^\infty, u_s\in Lip(X)$, and for a.e. $r\in (s, t)$
\begin{equation}
|{\bf \Gamma}_{2, r}(u_r)(g_r)\geq {1\over 2}\int_X \dot{\Gamma}_r)
(u_r)(g_r)dm_r+{1\over N}\left(\int_X \Delta_r u_rg_r dm_r\right)^2
 \end{equation}
 (“dynamic Bochner inequality” or “dynamic Bakry-Emery condition”)
where $u_r= P_{r, s}u_s$ and $g_r= P_{t, r}^*g_t$.
 
%
%
%
 

 \medskip
 
 \noindent{\bf Theorem 1.11 in \cite{Ko-Sturm18}}. Assume the time-dependent mm-space $(X, d_t, m_t, t\in I)$, is a
super-$(K, N)$-Ricci flow in the sense that for a.e. $t\in I$ and every $W_t$-geodesic $\mu^{a}, a\in [0, 1]$ in $\mathcal{P}$ 
 with $\mu^0, \mu^1\in {\rm Dom}(S)$ 
 \begin{equation}
\left. \partial_a^{+} S_t(\mu^a)\right|_{a=1-}-
\left. \partial_a^{-} S_t(\mu^a)\right|_{a=0+}
\geq -{1\over 2}\partial_t^{-}
W_{t-}^2(\mu^0, \mu^1)+{1\over N}|S_t(\mu^0)-S_t(\mu^1)|^2+KW_t^2(\mu^0, \mu^1).
 \end{equation}
Then for each 
$C\in \mathbb{R}$ the time-dependent mm-space 
$(X, \widetilde{d}_t, \widetilde{m}_t, t\in I)$
 is a super-$N$-Ricci flow if we let
 \begin{equation}
  \widetilde{d}_t=e^{-K\tau(t)}d_{\tau(t)}, \ 
  \ \widetilde{m}_t= m_{\tau(t)}, \ \ \tau(t)=-{1\over 2K}\log(C-2Kt),
 \end{equation}
 and $ \widetilde{I}=\{\tau(t): 2Kt<C\}$.
 
\medskip

\noindent{\bf Corollary  1.12 in \cite{Ko-Sturm18}}. For each $N\in (0, \infty)$  and $K\in \mathbb{R}$ the following are equivalent:\\
$({\bf I}_{K, N})$ For a.e. $t\in (0, T)$and every $W_t$ -geodesic $\mu^{a}, a\in [0, 1]$ in $\mathcal{P}$ 
 with $\mu^0, \mu^1\in {\rm Dom}(S)$
 \begin{equation}
\left. \partial_a^{+}S_t(\mu^a)\right|_{a=1-}-
\left. \partial_a^{-}S_t(\mu^a)\right|_{a=0+}
\geq -{1\over 2} \partial_t^{-}
W_{t-}^2(\mu^0, \mu^1)+KW_t^2(\mu^0, \mu^1)++{1\over N}|S_t(\mu^0)-S_t(\mu^1)|^2.
 \end{equation}
 $({\bf II}_{K, N})$ For all $0\leq s\leq t\leq T$ and $\mu, \nu
 \in \mathcal{P}$
\begin{equation}
W^2_s(\widehat{P}_{t, s}\mu, \widehat{P}_{t, s}\nu)
\leq W^2_{t}(\mu, \nu)-{2\over N}\int_s^t e^{-2Kr}[S_r( \widehat{P}_{t, r}\mu)-
S_r( \widehat{P}_{t, r}\nu)]^2dr.
 \end{equation}
$({\bf III}_{K, N})$ For all $u\in {\rm Dom}(E)$
and  $0\leq s\leq t\leq T$  
\begin{equation}
e^{2Kt}|\nabla_t (P_{t, s}u)|^2\leq e^{2Ks}P_{t, s}(|\nabla_s u|^2)-{2\over N}
\int_s^t e^{2Kr}(P_{t, r}\Delta_r P_{r, s}u)^2 dr.
 \end{equation}
$({\bf IV}_{K, N})$ For all $0\leq s\leq t\leq T$ and 
for all $u_s, g_t\in \mathcal{F}$  with $g_0\geq 0, 
g_t\in L^\infty, u_s\in Lip(X)$, and for a.e. $r\in (s, t)$
\begin{equation}
|{\bf \Gamma}_{2, r}(u_r)(g_r)\geq {1\over 2}\int \dot{\Gamma}_r)
(u_r)(g_r)dm_r+K\int_X \Gamma_r(u_r)g_r dm_r+{1\over N}\left(\int_X \Delta_r u_rg_r dm_r\right)^2
 \end{equation}
 where $u_r= P_{r, s}u_s$ and $g_r= P_{t, r}^*g_t$.
 
\medskip

\noindent{\bf Weighted case} (see {\bf Remark B.10 in \cite{Sturm18}}). Let $m_t$ be a family of 
reference measures on $X$, and $\phi_t\in \mathcal{A}$. Let 
$$\mu_t=e^{-\phi_t}m_t$$ 
be a family of weighted measure on $X$. We call $\phi_t$ the time dependent potential functions on 
$(X, d_t, \mu_t)$. 
Define $L_t$  as an operator on $\mathcal{A}$ by
$$
\int_X \Delta_t u vdm_t=-\int_X \Gamma_t (u, v)dm_t \ \ \ \forall u, v\in \mathcal{A},$$
and define similarly $L_t$ by replacing all $m_t$ 
by $\mu_t$. Then 
$$L_t=\Delta_t-\Gamma_t (\cdot, \phi_t)$$
and thus
$$\Gamma_{2}(L_t) = \Gamma_{2}(\Delta_t)+H_t \phi_t, \ \ \  Ric(L_t) = Ric(\Delta_t) +H_t \phi_t.$$
In particular, the family $(L_t)_{t\in (0,T)}$ defined by the family $(\Gamma_t, \phi_t)_{t\in (0, T)}$ 
is a super-Ricci flow if and only if
$$
\partial_t \Gamma_t\leq \Gamma_{2}(\Delta_t)+H_t \phi_t$$
which imposes no restriction on the evolution of the weights $\phi_t$. Each family of weight functions $(\phi_t)_{t\in (0, T)}$ provides a differential inequality for square field operators.

To end this subsection, let us recall the following result due to Sturm \cite{Sturm18}.

\begin{theorem}[Theorem 0.7 in \cite{Sturm18} ] The mm space $(X, d_t, m_t, t\in I=[0, T])$ 
 induced by a time dependent weighted $n$-dimesional Riemannaina manifold $(M, g_t, \phi_t, t\in I)$
is a super-$N$-Ricci flow if and only if $N\geq n$ and for all $t\in I$

\begin{eqnarray}
{1\over 2}{\partial g_t\over \partial t}+Ric_{g_t}+{\rm Hess}_{g_t}-{\nabla \phi_t\otimes \nabla \phi_t\over N-n}\geq 0.
\end{eqnarray}
In particular for $ N=n$ this requires $\phi_t$  to be constant. That is, $m_t=C_t vol_t$ for each $t\in I$.

\end{theorem}

\subsection{The notion of $(K, n, N)$-super Ricci flows}

To extend Perelman's $W$-entropy formula to super Ricci flows on metric measure spaces, we need to introduce some new definitions and notations.

Let $(X, d_t, m_t, t\in [0, T])$ be a family of time dependent RCD metric measure spaces.  Let $\{\phi_t, t\in [0, T]\} \subset \mathcal{A}$. Let
$$d\mu_t=e^{-\phi_t}dm_t$$ be a weighted measure on $(X, d_t, m_t)$. We call $\phi_t$ the potential function of $\mu_t$. Let 
$$L_t=\Delta_t-\Gamma_t(\phi_t, \cdot)=\Delta_t-\nabla_{{\bf g}_t}\phi_t \cdot \nabla_{{\bf g}_t}\cdot$$
be the time dependent Witten Laplacian on $(X, d_t, \mu_t)$.

The Dirichlet form with infinitesimal generator $\Delta_t$ on $(M, d_t, m_t)$  reads

$$
\mathcal{E}_{\Delta_t}(u, v)=\int_X \Delta_t u vdm_t=-\int_X \Gamma_t (u, v)dm_t \ \ \ \forall u, v\in \mathcal{A},$$
and the Dirichlet form with infinitesimal generator $L_t=\Delta_t-\Gamma_t(\phi_t, \cdot)$ on $(M, d_t, 
\mu_t)$  
reads

$$
\mathcal{E}_{L_t}(u, v)=\int_X L_t u vd\mu_t=-\int_X \Gamma_t (u, v)d\mu_t \ \ \ \forall u, v\in \mathcal{A}.$$

Following \cite{Gigili2015MAMS, BGZ, BG, Han2018A, Han2018B},  the tangent module $L^2(T(X, d_t, m_t))$ and the cotangent module $L^2(T^*(X, d_t, m_t))$ of an RCD$(K, N)$ space $(X, d_t, m_t)$ have been defined as $L^2$-normed modules. The pointwise inner product $\langle\cdot, \cdot\rangle: L^2(T^*(X, d_t, m_t))\times L^2(T^*(X, d_t, m_t))\rightarrow L^1(X, d_t, m_t) $ is defined by
$$ \langle df, dg\rangle={1\over 4}\left(|\nabla_t(f+g)|^2-|\nabla_t(f-g)|^2)\right)$$
for all $f, g\in W^{1, 2}(X, d_t, m_t)$. For any $g\in W^{1, 2}(X, d_t, m_t)$, its gradient $\nabla_t g$ is the unique 
element in $L^2(T(X, d_t, m_t))$ such that
$$\nabla_t g(df)=\langle df, dg\rangle,\ \ \ m_t-a.e. $$
for all $f\in W^{1, 2}(X, d_t, m_t)$. Therefore, $L^2(T(X, d_t, m_t))$ inherits a pointwise inner product $\langle\cdot, \cdot\rangle_{t}$ from the above inner product $\langle\cdot, \cdot\rangle_t$  on $L^2(T^*(X, d_t, m_t))$.  To keep the standard notation as used  in  Riemannian geometry, we  use  ${\bf g}_t$ to denote this inner product $\langle\cdot, \cdot\rangle_t$ on $L^2(T(X, d_t, m_t))$. 

The notion of local dimension $n$ of an RCD space $(X, d_t, m_t)$ is introduced in \cite{Han2018A, Han2018B} as follows: We say that $L^2(TM)$ is finitely generated if there is a finite family $v_1, ..., v_n$
spanning $L^2(T(X, d_t, m_t))$ on $(X, d_t, m_t)$, and locally finitely generated if there is a partition $\{E_i\}$ of $X$ such that 
$\left.L^2(TM)\right|_{E_i}$ is finitely finitely generated for every $i \in \mathbb{N}$. If $L^2(T(X, d_t, m_t))$ has a basis of cardinality $n_t$ on a Borel set $A\subset X$, we say that it has dimension $n_t$ on $A$,
or that its local dimension on $A$ is $n_t$. By See \cite{Han2018A, Han2018B, BS}, for each fixed $t$, the local dimension on $A$ is a global constant $n_t\in \mathbb{N}$. 
{\bf
From now on, we assume  the global geometric dimension of an RCD space $(X, d_t, m_t)$ is a constant $n$ which is independent of $t\in [0, T]$. }

The Hessian of a nice function $f\in \mathcal{A}$ is well-defined as in  \cite{Gigili2015MAMS, Sturm18, Han2018A, Han2018B}. It  is defined as  the unique bilinear map

$$\nabla^2_t f={\rm H}_t f: \{\nabla g: g\in {\rm Test} F(X)\}^2\mapsto L^0(X)$$
such that
$$2\nabla^2_t f(\nabla g, \nabla h)=\langle \nabla_t g, \nabla_t\langle \nabla_t f, \nabla_t h\rangle\rangle+
\langle \nabla_t h, \nabla_t\langle \nabla_t f, \nabla_t g\rangle\rangle
-\langle \nabla_t f, \nabla_t\langle \nabla g_t, \nabla_t h\rangle\rangle
$$
for any $g, h\in {\rm Test} F(X)$, where 
${\rm Test} F(X)=\{f\in D(\Delta_t)\cap L^\infty: |\nabla_t f|\in L^\infty, \Delta_t f \in W^{1, 2}(X, m_t)\}$ is the space of test functions. It is denoted by ${\rm H}_t f$ in \cite{Sturm18}  and  will be denoted by $\nabla^2_tf$ in this paper for keeping the standard notation as in Riemannian geometry. Note that
\begin{eqnarray*}
\nabla^2_t f(\nabla_tf, \nabla_t g)={1\over 2}\langle \nabla_t |\nabla_t f|^2, \nabla_t g\rangle.
\end{eqnarray*}

Similarly to \cite{Gigili2015MAMS, Sturm18, Han2018A, Han2018B}, we introduce 
\begin{equation*}
	{\bf \Gamma}_{2}(\Delta_t)(f, f):={1\over 2}{\bf \Delta}_t |\nabla_t f|^2-\langle \nabla_t f, \nabla_t \Delta_t f\rangle m_t
\end{equation*}
for all nice functions $f$ on RCD space $(X, d_t, m_t)$, where ${\bf \Delta}_t$ is the 
Laplacian in the distribution sense.  We define the measure 
valued Ricci curvature for the c Laplacian $\Delta_t$ on time dependent metric measure spaces as
\begin{equation*}
	{\bf Ric}(\Delta_t)(\nabla_t f, \nabla_t f):={\bf \Gamma}_{2}(\Delta_t)(f, f)-\|\nabla^2_t f\|_{\rm HS}^2 m_t.
\end{equation*}

Moreover, we introduce 

\begin{equation*}
	{\bf \Gamma}_{2}(L_t)(f, f):={1\over 2}L_t |\nabla_t f|^2-\langle \nabla_t f, \nabla_t L_t f\rangle \mu_t
\end{equation*}
and we have  the following distributional Bochner formula for the Witten Laplacian 
on time dependent metric measure space 
\begin{eqnarray}
{\bf \Gamma}_{2}(L_t)(f, f)=\|\nabla^2 f\|_{g_t, {\rm HS}}^2+{\bf Ric}_{\infty, n}
(L_t)(\nabla f, \nabla f).\label{DBHF}
\end{eqnarray}
Formally, we have 
\begin{eqnarray*}
{\bf Ric}_{\infty, n}(L_t)(\nabla f, \nabla f)=
{\bf Ric}(\Delta_t)(\nabla f, \nabla f)+(\nabla^2_t \phi_t)(\nabla f, \nabla f).
\end{eqnarray*}

For nice function $f$ whose Hessian  $\nabla^2_t f$ has finite Hilbert-Schmidt norm, i.e., 
$\|\nabla^2_t f\|_{\rm HS}<\infty$, the trace of $\nabla^2_t f$, denoted by ${\rm Tr}\nabla^2_t f$ in this paper 
as in Riemannian geometry, can be introduced by the same way as in Han \cite{Han2018A, Han2018B} as follows: Let $e_1, \ldots, e_n$ be a basis of the $L^2$-tangent module $L^2(TX,  d_t)$. Then
$${\rm Tr}\nabla^2_t f=\sum\limits_{1\leq i, j\leq n}\nabla^2_t f(e_i, e_j)\langle e_i, e_j\rangle.$$ 
In our notation, it reads as follows
$${\rm Tr}\nabla^2_t f=\langle\nabla^2_t f, {\bf g}_t\rangle.$$ 

We now introduce the following 

\begin{definition} The $N$-dimensional Bakry-Emery Ricci curvature {\bf measure} of the time dependent Witten Laplacian 
$$L_t=\Delta_t-\nabla_{{\bf g}_t}\phi_t\cdot\nabla_{{\bf g}_t}\cdot$$ on 
an $n$-geometric dimensional RCD space $(X, d_t, m_t)$ is defined  as follows

\begin{equation*}
	{\bf Ric_{N, n}}(L_t)(\nabla_t f, \nabla_t f):={\bf Ric}_{\infty, n}(\nabla_t f, \nabla_t f)
	-{[{\bf g}_t(\nabla_t \phi_t, \nabla_t f)]^2\over N-n}.
	\end{equation*}
Formally, we have
\begin{eqnarray*}
{\bf Ric}_{N, n}(L_t)(\nabla f, \nabla f):=
{\bf Ric}(\Delta_t)(\nabla f, \nabla f)+(\nabla^2_t \phi_t)(\nabla f, \nabla f) -{[{\bf g}_t(\nabla_t \phi_t, \nabla_t f)]^2\over N-n}.
\end{eqnarray*}
\end{definition}

We now introduce the notion of the $(K, n, N)$-super Ricci flow on time dependent metric measure spaces.

\begin{definition} We call an $n$-dimensional time dependent metric measure space $(X, d_t, {\bf g}_t, m_t, \phi_t, t\in [0, T])$ a 
$(K, n, N)$-super Ricci flow  if 
\begin{eqnarray*}
{1\over 2}{\partial {\bf g}_t\over \partial t}+{\bf Ric_{N, n}}(L_t)\geq K{\bf g}_t, \ \ \ \ \forall ~ t\in [0, T],
\end{eqnarray*}
where $K\in \mathbb{R}$ and $N\in [n, \infty]$ are two constants. In particular, we call an $n$-dimemnsional time dependent metric measure space $(X, d_t, {\bf g}_t, m_t, \phi_t, t\in [0, T])$ a 
$(K, n, N)$-Ricci flow  if 
\begin{eqnarray*}
{1\over 2}{\partial {\bf g}_t\over \partial t}+{\bf Ric_{N, n}}(L_t)=K{\bf g}_t, \ \ \ \ \forall ~ t\in [0, T].
\end{eqnarray*}
\end{definition}

Recall that when $(X, d, {\bf g}, m, \phi, \mu)$ is an $n$-dimensional RCD metric measure space satisfying the distributional Bochner formula 
\begin{eqnarray}
{\bf \Gamma}_{2}(L)(f, f)=\|\nabla^2 f\|_{{\bf g}, {\rm HS}}^2+{\bf Ric}_{\infty, n}
(L)(\nabla f, \nabla f).\label{DBHF0}
\end{eqnarray}
and with 
$${\bf Ric}_{N, n}(L)\geq K{\bf g},$$ 
where $N\geq n$ and $K\in \mathbb{R}$ are two constants, we call it an RCD$(K, n, N)$ space, which has been introduced in our previous paper with Zhang \cite{Li-Zhang2025}. 
%
Obviously, an RCD$(K, n, N)$ space is indeed a stationary $(K, n, N)$-super Ricci flow on metric measure spaces.

Similarly to Perelman \cite{P1} and S. Li-Li \cite{LiLi2015PJM, LiLi2018SCM, LiLi2018JFA, LiLi2020JGA}, we introduce 
\begin{definition}
The conjugate heat equation on a family of time dependent metric measure $(X, d_t, m_t)$ reads 
\begin{eqnarray}\label{conjugateEq}
 {d\over dt}\left(e^{-\phi_t}dm_t\right)=0.
\end{eqnarray}
Equivalently, $(\phi_t, {\bf g}_t, t\in [0, T])$ satisfies 
\begin{eqnarray}\label{conjugateEq}
\partial_t \phi_t={1\over 2}{\rm Tr}\left(\partial_t {\bf g}_t\right).
\end{eqnarray}

\end{definition}

In the case $N=n$, the notion of the $(K, n, N)$-super Ricci flow is indeed the $K$-super Ricci flow in geometric analysis
\begin{eqnarray*}
{1\over 2}{\partial {\bf g}_t\over \partial t}+Ric_{{\bf g}_t}\geq K{\bf g}_t, \ \ \ \ \forall ~ t\in [0, T], 
\end{eqnarray*}
and in the case $N=\infty$, the $(K, \infty)$-super Ricci flow equation reads
\begin{eqnarray*}
{1\over 2}{\partial  {\bf g}_t\over \partial t}+{\bf Ric}(L_t)\geq K{\bf g}_t, \ \ \ \ \forall ~ t\in [0, T]. \label{KKK}
\end{eqnarray*}
In view of this,  the modified 
Ricci flow introduced by Perelman in \cite{P1} is indeed the $(0, \infty)$-Ricci flow together with the conjugate heat equation
\begin{eqnarray*}
{\partial {\bf g}_t\over \partial t}&=&-2{\bf Ric}(L_t),\\ \ \ \ {\partial\phi_t\over \partial t}&=&{1\over 2}{\rm Tr} \left({\partial {\bf g}_t\over \partial t}\right).
\end{eqnarray*}
where $R_{{\bf g}_t}={\rm Tr} {\bf Ric}_{{\bf g}_t}$ is the scalar curvature of the Riemannian metric ${\bf g}_t$.
More precisely, the Perelman modified 
Ricci flow reads (see \cite{P1})
\begin{eqnarray*}
{\partial {\bf g}_t\over \partial t}&=&-2\left({\bf Ric}_{{\bf g}_t}+\nabla_{{\bf g}_t}^2 \phi_t\right),\\
 \ \ \ {\partial\phi_t\over \partial t}&=&-\Delta_t\phi_t- R_{{\bf g}_t}.
\end{eqnarray*}

\section{$W$-entropy formulas on super Ricci flows on mm spaces}

In this section, we first state the dissipation formulas of the Boltzmann-Shannnon $H$-entropy associated with  the heat equation $\partial_t u=L u$ on metric measure spaces with time dependent metrics and potentials. Then we state the main results of this paper. The proofs will be given in Section 5. 

\subsection{$H$-entropy formulas on time dependent metric measure spaces}

We now state the $H$-entropy dissipation formulas  on closed metric measure spaces with time dependent metrics and potentials. In the Riemannian case, these formulas were proved by S. Li and the author in \cite{LiLi2015PJM}. 

\begin{theorem}\label{Th-B1}  Let $(X, d(t), {\bf g}(t), m_t, \phi_t)$ be a family of  time dependent closed RCD spaces. Let
$$d\mu_t=e^{-\phi_t}dm_t.$$
Suppose that $d\mu_t$ is independent of $t\in [0, T]$, i.e., $({\bf g}(t), m_t, \phi_t)$ satisfies the conjugate equation 
\begin{eqnarray}\label{conjugateHE5}
\frac{\partial \phi_t} {\partial t}={1\over 2}{\rm Tr}\left(
\frac{\partial {\bf g}_t}{\partial t}\right).
\end{eqnarray}
Let $u$ be a positive solution to the heat equation  $\partial_t u=L_t u$ associated with the time dependent Witten Laplacian $L_t=\Delta_{g_t}-\nabla_{g_t} \phi_t\cdot\nabla_{g_t}$. Suppose that 
$u\in W^{1, 2}(X,\mu)\cap D(L)\cap L^\infty(X, \mu)$ with $Lu\in L^\infty(X, \mu)$. Let
$$H(u)=-\int_X u\log u d\mu$$ 
be the Boltzmann-Shannon entropy associated with the time dependent heat equation $\partial_t u=L_t u$. Then
\begin{eqnarray}
{d\over dt} H(u(t)) &=&\int_X {|\nabla u|^2\over u}d\mu,\label{1H}\\
{d^2\over d t^2} H(u(t))&=&-2\int_X \left[|\nabla^2\log u|^2+\left({1\over 2}{\partial {\bf g}\over \partial t}+{\bf Ric_{\infty, n}}(L_t)\right)(\nabla \log u, \nabla \log u)\right]u d\mu.\label{2H}
\end{eqnarray}

\end{theorem}
\medskip

As an easy consequence of Theorem \ref{Th-B1}, we have

\begin{theorem}\label{Th-B2}  Let $(X, d_t, {\bf g}_t, m_t, \phi_t, t\in [0, T])$ be a family of $n$-dimensional closed metric measure space  with time dependent metrics and potentials. 
Suppose that $({\bf g}_t, \phi_t)$ is a $(K, n, \infty)$-super Ricci flow and satisfies the conjugate equation, i.e.,
\begin{eqnarray*}
{\partial {\bf g}_t\over \partial t}&\geq &-2{\bf Ric_{\infty, n}}(L_t),\\
\frac{\partial\phi_t} {\partial t}&=&{1\over 2}{\rm Tr}\left(
\frac{\partial {\bf g}_t}{\partial t}\right).
\end{eqnarray*}
Let $u$ be a positive solution to the heat equation $\partial_t u=L_tu$ for the time dependent Witten Laplacian $L_t=\Delta_t-\nabla_t\phi_t\cdot\nabla_t$.  Suppose that 
$u\in W^{1, 2}(X,\mu)\cap D(L)\cap L^\infty(X, \mu)$ with $Lu\in L^\infty(X, \mu)$.
Let
$$H(u)=-\int_X u\log u d\mu$$
be the associated Boltzmann-Shannon entropy. Then
\begin{eqnarray*}
{d\over dt} H(u(t)) \geq 0,
\end{eqnarray*}
and
\begin{eqnarray*}
{d^2\over d t^2} H(u(t))\leq 0.
\end{eqnarray*}

\end{theorem}

\medskip

\subsection{$W$-entropy formulas on super Ricci flows on mm spaces}

We now state the main results of this paper.  Our first result extends Perelman's $W$-entropy formula to 
Sturm's $(K, N)$-super Ricci flows on closed metric measure spaces.  

\begin{theorem} \label{W-Sturm}  
Let $(X, d_t,  {\bf g}(t), m_t, \phi_t, t\in [0, T])$ be  a closed $(K, N)$ super Ricci flow on mm spaces with the conjugate equation \eqref{conjugate}, 
where $N\geq 1$ and $K\in \mathbb{R}$. 
Let $u$ be a positive solution to the heat equation $\partial_t u=L u$. Suppose that $u\in W^{1, 2}(X,\mu)\cap D(L)\cap L^\infty(X, \mu)$ with $Lu\in L^\infty(X, \mu)$. Then

\begin{equation}
\begin{aligned}
\frac{d}{d t} W_{N, K}(u) = -2t\int_X \Big[{1\over 2}{\partial {\bf g}_t\over \partial t}+{\bf \Gamma_2}(\log u, \log u )+\left(\frac{1}{t}-K\right)\Delta \log u+{N\over 4}\Big(\frac{1}{t}-K \Big)^2 -K|\nabla \log u |^2\Big]ud\mu. \label{W-Gamma2A}
\end{aligned}
\end{equation}
Moreover, we have
\begin{equation}
	\frac{d}{dt}W_{N,K}(u)\leq -\frac{2t}{N}\int_X u\Big(L \log u+\frac{N}{2t}-\frac{NK}{2}\Big)^2d\mu.\label{WKN}
\end{equation}
In particular, we have
\begin{equation*}
	\frac{d}{dt}W_{N,K}(u)\leq 0.
\end{equation*}
\end{theorem}

The following result extends the $W$-entropy formula due to S. Li and the author of this paper \cite{LiLi2015PJM} from the  $(0, m)$-super Ricci flows on smooth Riemannian manifolds to $(0, n, N)$-super Ricci flows on metric measure spaces.  

\begin{theorem} \label{mmW1} Let $(X, d(t), {\bf g}(t), m_t, \phi_t)$ be a family of  time dependent closed RCD spaces. Let
$$d\mu_t=e^{-\phi_t}dm_t.$$
Suppose that $d\mu_t$ is independent of $t\in [0, T]$, i.e., $({\bf g}(t), m_t, \phi_t)$ satisfies the conjugate equation \eqref{conjugateHE5}. Let $u$ be a positive solution to the heat equation  $\partial_t u=L_t u$ associated with the time dependent Witten Laplacian $L_t=\Delta_{g_t}-\nabla_{g_t} \phi_t\cdot\nabla_{g_t}$. 
Suppose that $u\in W^{1, 2}(X,\mu)\cap D(L)\cap L^\infty(X, \mu)$ with $Lu\in L^\infty(X, \mu)$.  Let
$$H(u)=-\int_X u\log u d\mu$$ 
be the Boltzmann-Shannon entropy associated with the time dependent heat equation $\partial_t u=L_t u$.  Let 

\begin{eqnarray*}
H_N(u, t)=-\int_X u\log u d\mu-{N\over 2}(1+\log(4\pi t)).
\end{eqnarray*}
Define
\begin{eqnarray*}
W_N(u, t)={d\over dt}(tH_N(u)).
\end{eqnarray*}
Then
\begin{eqnarray}
W_N(u, t)=\int_X \left[t|\nabla f|^2+f-N\right]ud\mu, \label{NWmm1}
\end{eqnarray}
and
\begin{eqnarray}
& &{d\over dt}W_N(u, t)=-2t\int_X \left\|\nabla^2 \log u+{{\bf g}_t\over 2t}\right\|_{\rm HS}^2ud\mu-{2t\over N-n}\int_X \left(\nabla \phi\cdot \nabla \log u-{N-n\over 2t}\right)^2  ud\mu\nonumber\\
& &\hskip4cm -2t\int_X \left({1\over 2}{\partial {\bf g}_t\over \partial t}+{\bf Ric_{N, n}}(L_t)\right)(\nabla \log u, \nabla \log u)ud\mu.\label{NWmm2}
\end{eqnarray}
As a consequence, we have
\begin{equation}
	\frac{d}{dt}W_{N}(u, t)\leq -\frac{2t}{N}\int_X u\Big(L\log u+\frac{N}{2t}\Big)^2d\mu.\label{WK0}
\end{equation}
In particular, it holds
\begin{equation*}
	\frac{d}{dt}W_{N}(u)\leq 0.
\end{equation*}

In particular, if $\{{\bf g}_t, \phi_t, t\in (0, T]\}$ is a $(0, n, N)$-super Ricci flow and satisfies the conjugate equation $(\ref{conjugateHE5})$, then $W_N(u, t)$ is decreasing in $t\in (0, T]$, i.e.,
\begin{eqnarray*}
{d\over dt}W_N(u, t)\leq 0, \ \ \ \forall t\in (0, T].
\end{eqnarray*}
Moreover, the left hand side in $(\ref{NWmm2})$ identically  equals to zero  on $(0, T]$ if and only if $(X, g(t), \phi(t), t\in (0, T])$ is a $(0, n, N)$-Ricci flow in the sense that

\begin{equation*}
\nabla^2 \log u+{{\bf g}\over 2t}=0, \ \ \ 
{1\over 2}{\partial {\bf g}\over \partial t}+{\bf Ric_{N, n}}(L_t)=0,\ \ \ \nabla \phi\cdot \nabla \log u={N-n\over 2t},
	\end{equation*}
and
\begin{eqnarray*}{\partial \phi\over \partial t}&=&{1\over 2} {\rm Tr}\left( {\partial g\over \partial t}\right).
\end{eqnarray*}
\end{theorem}

In general case $K\neq 0$, the following result extends the $W$-entropy formula due to S. Li and the author of this paper \cite{LiLi2015PJM} from the  $(K, m)$-super Ricci flows on smooth Riemannian manifolds to $(K, n, N)$-super Ricci flows on metric measure spaces.  

\begin{theorem}\label{mmW2} Let $(X, d(t), {\bf g}(t), m_t, \phi_t)$ be a family of  time dependent closed RCD spaces. Let
$$d\mu_t=e^{-\phi_t}dm_t.$$
Suppose that $d\mu_t$ is independent of $t\in [0, T]$, i.e., $({\bf g}(t), m_t, \phi_t)$ satisfies the conjugate equation \eqref{conjugateHE5}. Let $u$ be a positive solution to the heat equation  $\partial_t u=L_t u$ associated with the time dependent Witten Laplacian $L_t=\Delta_{g_t}-\nabla_{g_t} \phi_t\cdot\nabla_{g_t}$. Suppose that $u\in W^{1, 2}(X,\mu)\cap D(L)\cap L^\infty(X, \mu)$ with $Lu\in L^\infty(X, \mu)$. 
Let
$$H(u)=-\int_X u\log u d\mu$$ 
be the Boltzmann-Shannon entropy associated with the time dependent heat equation $\partial_t u=L_t u$.  Let 
\begin{eqnarray}
H_{N,K}(u, t)=-\int_X u\log u d\mu-{N\over 2}(1+\log(4\pi t))-\frac N2Kt\Big(1+\frac16Kt\Big),  \label{mmHNK}
\end{eqnarray}
and define
\begin{align}
W_{N,K}(u, t)={d\over dt}(tH_{m,K}(u)). \label{mmWNK}
\end{align}
Then
\begin{eqnarray*}
W_{N,K}(u, t)=\int_X\left[t|\nabla f |^2+f-N\Big(1+\frac12Kt\Big)^2\right]ud\mu,\label{WNK}
\end{eqnarray*}
and
\begin{align}\label{WnNK}
{d\over dt}W_{N,K}(u, t)&=-2 t\int_X\Big\|\nabla^2 f-\left(\frac1{2t}+\frac{K}{2}\right){\bf g}\Big\|_{\rm HS}^2u d\mu\nonumber\\
&\hskip1cm -\frac{2t}{m-n}\int_X\left(\nabla \phi\cdot\nabla \log u-(N-n)\Big(\frac1{2t}+\frac K2\Big)\right)^2u d\mu\nonumber\\
&\hskip1.5cm -2 t\int_X\left({1\over 2}{\partial {\bf g}_t\over \partial t}+{\rm Ric}_{N,n}(L)+K{\bf g}\right)(\nabla f, \nabla f) ud\mu.
\end{align}
In particular, if $(X, g(t), \phi(t), t\in (0, T])$ is a $(-K, n, N)$-super Ricci flow on metric measure space and satisfies the conjugate equation $(\ref{conjugateHE5})$, then $W_{N,K}(u, t)$ is decreasing in $t\in (0, T]$, i.e.,
\begin{eqnarray*}
{d\over dt}W_{N,K}(u, t)\leq 0, \ \ \ \forall t\in (0, T].
\end{eqnarray*}
Moreover, the left hand side in $(\ref{WnNK})$ identically  equals to zero on $(0, T]$ if and only if $(M, g(t), \phi(t), t\in (0, T])$ is a $(-K, n, N)$-Ricci flow in the sense that

\begin{equation*}
\nabla^2 \log u+{1\over 2}\left({1\over t}-K\right){\bf g}=0, \ \ \ 
{1\over 2}{\partial {\bf g}\over \partial t}+{\bf Ric_{N, n}}(L_t)=K{\bf g}, \ \ \ 
\nabla \phi\cdot\nabla \log u={N-n\over 2}\left({1\over t}-K\right),
	\end{equation*}
and
\begin{eqnarray*}{\partial \phi\over \partial t}&=&{1\over 2} {\rm Tr}\left( {\partial g\over \partial t}\right).
\end{eqnarray*}
\end{theorem}

\medskip

\subsection{$W$-entropy formulas on static RCD spaces}

The above results can be also regarded as natural extensions of the corresponding $W$-entropy formulas on 
RCD spaces with fixed metrics and measure. 

\begin{theorem} [Kuwada-Li \cite{KL}, Li-Zhang \cite{Li-Zhang2025}]\label{KL0}
Let $(X, d, \mu)$ be a metric measure space satisfying the 
 RCD$(0, N)$-condition. Then
\begin{equation*}
	\frac{d}{dt}W_N(u)\leq 0.
\end{equation*}

Moreover, ${d\over dt}W_{N} (u(t))=0$ holds at some $t=t_*>0$ for the  fundamental solution of the heat equation $\partial_t u=\Delta u$ if and only if $(X, d, \mu)$ is one of the following rigidity models:

(i) If $N \geq 2$,  $(X, d, \mu)$ is $(0, N-1)$-cone over an $\operatorname{RCD}(N-2, N-1)$ space and $x$ is the vertex of the cone.

(ii) If $N<2$,  $(X, d, \mu)$ is isomorphic to either $\left([0, \infty), d_{\text {Eucl }}, x^{N-1} \mathrm{~d} x\right)$ or  $\left(\mathbb{R}, d_{\mathrm{Eucl}},|x|^{N-1} \mathrm{~d} x\right)$, where $d_{\text {Eucl }}$ is the canonical Euclidean distance. 

In each of the above cases, $W_N(u(t))$ is a constant on $(0, \infty)$, the Fisher information  $I(u(t))$ is given by $I(u(t))={N\over 2t}$ for all $t\in (0, \infty)$, and there exists some $x_0\in M$ such that
$$\Delta d^2(\cdot, x_0)=2N.$$ 

\end{theorem}
		 
Recently, in a joint paper with Zhang \cite{Li-Zhang2025}, we proved the following $W$-entropy formula on 
RCD$(K, n, N)$ spaces. 
		 
\begin{theorem} [Li-Zhang  \cite{Li-Zhang2025}] \label{WnNK}  
Let $(X, d, \mu)$ be  an RCD$(K, n, N)$ space, where $n\in \mathbb{N}$, $N\geq n$ and $K\in \mathbb{R}$. Let $u$ be a positive solution to the heat equation $\partial_t u=\Delta u$ satisfying reasonable growth condition as required in  \cite{Li-Zhang2025}. Then
\begin{equation*}
\begin{aligned}
\frac{d}{d t} W_{N, K}(u)
&= -2t\int_X \left\|\nabla^2 \log u+{1\over 2}\left({1\over t}-K\right){\bf g}\right\|_{\rm HS}^2u d\mu\\	
& \hskip1cm -2t\int_X \left({\bf Ric_{N, n}}(L_t)-K{\bf g}\right)(\nabla \log u, \nabla \log u)u d\mu\\		
& \hskip1.5cm -{2t\over N-n}\int_X \Big[\Gamma_t(\phi, \log u)-{N-n\over 2}\left({1\over t}-K\right)\Big]^2u d\mu. \label{WfnKN}
\end{aligned}
\end{equation*}
Moreover, we have
\begin{equation*}
	\frac{d}{dt}W_{N,K}(u)\leq -\frac{2t}{N}\int_X \Big[L_t\log u+{N\over 2}\left({1\over t}-K\right)\Big]^2u d\mu.
\end{equation*}
In particular,  ${d\over dt}W_{N,K}(u)\leq0$, and ${d\over dt}W_{N,K} (u)=0$ holds at some $t>0$ if and only if at this $t$,

\begin{equation*}
\nabla^2 \log u+{1\over 2}\left({1\over t}-K\right){\bf g}=0, \ \ \ 
{\bf Ric_{N, n}}(L_t)=K{\bf g},\ \ \ 
\nabla\phi\cdot\nabla \log u={N-n\over 2}\left({1\over t}-K\right).
	\end{equation*}

\end{theorem}

In the particular case $K=0$, we have the following 

\begin{theorem} [Li-Zhang \cite{Li-Zhang2025}, see also Brena \cite{Brena2025}] \label{WnNK=0}  
Let $(X, d, \mu)$ be  an RCD$(0, n, N)$ space, where $n\in \mathbb{N}$ and $N\geq n$. Let $u$ be a positive solution to the heat equation $\partial_t u=\Delta u$ satisfying reasonable growth condition as required in \cite{Li-Zhang2025}. Then
\begin{equation*}
\begin{aligned}
\frac{d}{d t} W_{N}(u)
&= -2t\int_X \Big[\left\|\nabla^2 \log u+{{\bf g}\over 2t}\right\|_{\rm HS}^2+{\bf Ric_{N, n}}(L_t)(\nabla \log u, \nabla \log u)\Big]u d\mu\\		
& \hskip2cm -{2t\over N-n}\int_X \Big[\Gamma_t(\phi_t, \log u)-{N-n\over 2t}\Big]^2u d\mu.\label{WfnKN}
\end{aligned}
\end{equation*}
In particular, on any RCD$(0, n, N)$ space, we have
\begin{equation*}\label{dW=0}
	\frac{d}{dt}W_{N}(u)\leq -\frac{2t}{N}\int_X \Big[L_t \log u+{N\over 2t}\Big]^2u d\mu \leq 0.
\end{equation*}
%
\end{theorem}

\section{Proofs of theorems}

\subsection{Proof of Theorem \ref{Th-B1} and Theorem \ref{Th-B2} }
The proof is similar the one in \cite{LiLi2015PJM}. On closed metric measure space, direct calculation yields
\begin{eqnarray*}
{\partial\over \partial t} H(u(t))=-\int_X \partial_t u(\log u+1)d\mu=-\int_X Lu (\log u+1)d\mu.
\end{eqnarray*}
Integrating by parts yields
\begin{eqnarray*}
{\partial\over \partial t} H(u, t)=\int_X |\nabla\log u|^2_{g(t)} ud\mu=\int_X {|\nabla u|_{g(t)}^2\over u} d\mu.
\end{eqnarray*}

Furthermore, as ${\partial\over \partial t} (d\mu)=0$, we have
\begin{eqnarray}
{\partial^2\over \partial t^2}H(u, t)&=&\int_X {\partial\over \partial t}(|\nabla\log u|^2_{g(t)} u)d\mu\nonumber\\
&=&\int_X \left[{\partial \over \partial t}g^{ij}\nabla_i\log u\nabla_j\log u\right]u d\mu+\int_X {\partial \over \partial t}\left[{|\nabla u|^2\over u}\right]_{\rm g(t)\ fixed}d\mu\nonumber\\
&=&\int_X \left[-{\partial \over \partial t}g_{ij}\nabla_i\log u\nabla_j\log u\right]u d\mu+\int_X {\partial \over \partial t}\left[{|\nabla u|^2\over u}\right]_{\rm g(t)\ fixed}d\mu\nonumber\\
&=&\int_X \left(-{\partial g\over \partial t}(\nabla\log u, \nabla u)+{\partial \over \partial t}\left[{|\nabla u|^2\over u}\right]_{\rm g(t)\ fixed}\right)d\mu,\label{aaa}
\end{eqnarray}
where $[\cdot]_{\rm g(t)\ fixed}$ means that the quantity $|\nabla u|^2$ in $[\cdot]$ is defined under a fixed metric $g(t)$, 
and in the third step we use the facts  $|\nabla \log u|^2=g^{ij}\nabla_i\log u\nabla_j\log u$ implies
$$\partial_t g^{ij}=-\partial_t g_{ij}.$$

On the other hand, on closed mm space with fixed metric $g(t)$ and with time independent measure $d\mu$, 
similarly to the proof of  the entropy dissipation formula in \cite{BE1985, Li2012, Li2016SPA, Li-Zhang2025}, we have
\begin{eqnarray}
\int_X
& &{\partial \over \partial t}\left[{|\nabla u|^2\over u}\right]_{\rm g(t)\ fixed}d\mu
=\int_X
{\partial \over \partial t}\left[|\nabla \log u|^2 u\right]_{\rm g(t)\ fixed}d\mu\nonumber\\
&=&\int_X\left[2\langle \nabla \log u, \nabla {L u\over u}\rangle  u+|\nabla \log u|^2 Lu\right]_{\rm g(t)\ fixed}d\mu\nonumber\\
&=&\int_X\left[2\langle \nabla \log u, \nabla (L \log u+|\nabla \log u|^2)\rangle  u+|\nabla \log u|^2 Lu\right]_{\rm g(t)\ fixed}d\mu\nonumber\\
&=&\int_X 2\langle \nabla \log u, \nabla L \log u\rangle ud\mu+\int_M \left[ 2\langle\nabla u, 
\nabla |\nabla \log u|^2 
\rangle + L|\nabla \log u|^2 u\right]_{\rm g(t)\ fixed}d\mu\nonumber\\
&=&\int_X 2\langle \nabla \log u, \nabla L \log u\rangle ud\mu+\int_M \left[ -2L u |\nabla \log u|^2 
+ L|\nabla \log u|^2 u\right]_{\rm g(t)\ fixed}d\mu\nonumber\\
&=&\int_X\left[2\langle \nabla \log u, \nabla L \log u\rangle-L|\nabla \log u|^2 
\right]_{\rm g(t)\ fixed}ud\mu\nonumber\\
&=&-2\int_X \Gamma_2(L_t)(\log u, \log u)ud\mu\nonumber.
\end{eqnarray}
By the distributional Bochner formula \eqref{DBHF}, we have
\begin{eqnarray}
\int_X
{\partial \over \partial t}\left[{|\nabla u|^2\over u}\right]_{\rm g(t)\ fixed}d\mu
=-2\int_X \left[\|\nabla^2\log u\|^2_{\rm HS}+{\bf Ric_{\infty, n}}(L_t)(\nabla \log u, \nabla \log u)\right]ud\mu.\label{bbb}
\end{eqnarray}
Combining $(\ref{aaa})$ and $(\ref{bbb})$, we complete the proofs of Theorem \ref{Th-B1} and Theorem \ref{Th-B2} . \hfill $\square$

\medskip

\subsection{Proof of Theorem \ref{W-Sturm} }
 \medskip

By the definition formula of $W_{K, N}$ and the entropy dissipation identities $(\ref{1H})$ and $(\ref{2H})$ in Theorem \ref{Th-B1}, we have  
\begin{equation}
\begin{aligned}
&\frac{d}{d t} W_{N, K}(u) =tH''+2H'-\frac{N}{2t}+N K\left(1-\frac{K t}{2}\right)\\
&=-2t\int_X \left[{1\over 2}{\partial {\bf g}_t\over \partial t}+{\bf \Gamma_2}(L)(\log u, \log u )\right]ud\mu
+2\int_X|\nabla\log u |^2 ud\mu-\frac{N}{2t}+N K\left(1-\frac{K t}{2}\right)\\
&= -2t\int_X \Big[{1\over 2}{\partial {\bf g}_t\over \partial t}+{\bf \Gamma_2}(L)(\log u, \log u )
+\left(\frac{1}{t}-K\right)L \log u+{N\over 4}\Big(\frac{1}{t}-K \Big)^2 -K|\nabla \log u |^2\Big]ud\mu. \label{W-Gamma2A}
\end{aligned}
\end{equation}
Under Sturm's $(K, N)$-super Ricci flow, the weak Bochner inequality holds in the sense of 
distribution
\begin{equation}
\begin{aligned}
{1\over 2}{\partial {\bf g}_t\over \partial t}+{\bf \Gamma_2}(L)(\log u, \log u )\geq {|L\log u|^2\over N}+K|\nabla 
\log u|^2. \label{weakBochner}
\end{aligned}
\end{equation}
Therefore

\begin{equation*}
\begin{aligned}
\frac{d}{d t} W_{N, K}(u) 
&\leq -2t\int_X \Big[ {|L\log u|^2\over N}+K |\nabla \log u|^2+\left(\frac{1}{t}-K\right)L \log u+{N\over 4}\Big(\frac{1}{t}-K \Big)^2 -K|\nabla \log u |^2\Big] ud\mu\\
			&= -\frac{2t}{N} \int_X\left[L \log u +{N\over 2}\left(\frac{1}{t}-K\right)\right]^2ud\mu.
\end{aligned}
\end{equation*} This finishes the proof of Theorem \ref{W-Sturm}. \hfill $\square$

\medskip

As a corollary, we have the following result which was originally proved by Kuwada and Li \cite{KL}. 

\begin{corollary} (i.e., Theorem \ref{KL0})
Let $(X, d, \mu)$ be an 
 RCD$(0, N)$ space and $u$ be a positive solution to the heat equation $\partial_t u=\Delta u$. Then
\begin{equation*}
	\frac{d}{dt}W_N(u)\leq -\frac{2t}{N}\int_X u\Big(\Delta \log u+\frac{N}{2t}\Big)^2d\mu.\label{WW}
\end{equation*}
In particular, we have
\begin{equation*}
	\frac{d}{dt}W_N(u)\leq 0.
\end{equation*}
\end{corollary}

\subsection{Proof of Theorem \ref{mmW1} and Theorem \ref{mmW2}}

Under the condition that the Riemannian Bochner formula $(\ref{RBFnN})$ holds, we can prove Theorem 
\ref{mmW1} and Theorem \ref{mmW2} by the same argument as used in Li \cite{Li2012} and S. Li-Li \cite{LiLi2015PJM} for the $W$-entropy formulas   on Riemannian manifolds with CD$(K, m)$-condition and closed $(K, m)$-super Ricci flows. See also Li-Zhang \cite{Li-Zhang2025} in which the authors proved the $W$-entropy formula on RCD$(K, n, N)$ spaces. For the completeness of the paper, we  reproduce the proof as follows.  By $(\ref{W-Gamma2A})$, we have
	\begin{equation*}
\begin{aligned}
\frac{d}{d t} W_{N, K}(u)
&= -2t\int_X \Big[{1\over 2}{\partial {\bf g}_t\over \partial t}+{\bf \Gamma_2}(L)(\log u, \log u )+\left(\frac{1}{t}-K\right)L \log u
+{N\over 4}\Big(\frac{1}{t}-K \Big)^2 -K|\nabla \log u |^2\Big]ud\mu\\
&=-2t\int_X\left[{1\over 2}{\partial {\bf g}_t\over \partial t}+\left\|\nabla^2 \log u\right\|^2_{\rm HS}+\operatorname{Ric}(L)(\nabla \log u, \nabla \log u ) \right]u d\mu\\
& \hskip2cm +2\int_X|\nabla\log u |^2 ud\mu-\frac{N}{2t}+N K\left(1-\frac{K t}{2}\right).
\end{aligned}
\end{equation*}
Splitting 
$$L \log u={\rm Tr}\nabla^2 \log u+(L-{\rm Tr} \nabla^2) \log u,$$
we have
\begin{equation*}
\begin{aligned}
\frac{d}{d t} W_{N, K}(u) &= -2t\int_X \Big[{1\over 2}{\partial {\bf g}_t\over \partial t}+{\bf \Gamma_2}(L)(\log u, \log u )
+\left(\frac{1}{t}-K\right){\rm Tr}\nabla^2 \log u+{n\over 4}\Big(\frac{1}{t}-K \Big)^2-K|\nabla \log u |^2\Big]ud\mu\\
&\hskip1cm -2t\int_X \left[\left(\frac{1}{t}-K\right)(L -{\rm Tr}\nabla^2)\log u 
+{N-n\over 4}\Big(\frac{1}{t}-K \Big)^2 \right]ud\mu.
\end{aligned}
\end{equation*}
By assumption,  the Riemannian Bochner formula $(\ref{RBFnN})$ holds in the sense of distribution. Thus
\begin{eqnarray*}
& &{\bf \Gamma_2}(L)(\log u, \log u)+\left({1\over t}-K\right){\rm Tr}\nabla^2 \log u+{n\over 4}\left({1\over t}-K\right)^2\\
& &\hskip1.5cm =\left\|\nabla^2 \log u+{1\over 2}\left({1\over t}-K\right){\bf g}\right\|^2_{\rm HS}+\operatorname{\bf Ric_{N, n}}(L)(\nabla\log u, \nabla\log u)+{|(L-{\rm Tr}\nabla^2)\log u|^2\over N-n}.\end{eqnarray*}
This yields 
\begin{equation*}
\begin{aligned}
\frac{d}{d t} W_{N, K}(u, t)&=  -2t \int_X \left[\left\|\nabla^2 \log u+{1\over 2}\left({1\over t}-K\right){\bf g}\right\|^2_{\rm HS}+\left({1\over 2}{\partial {\bf g}_t\over \partial t}+\operatorname{\bf Ric_{N, n}}(L)-K{\bf g}\right)(\nabla \log u, \nabla \log u) \right]u d \mu\\
& \hskip0.5cm -2t\int_X \left[{|(L-{\rm Tr}\nabla^2)\log u)|^2\over N-n} +\left({1\over t}-K\right)(L-{\rm Tr}\nabla^2) \log u+{N-n\over 4}\left(\frac{1}{ t}-K\right)^2 \right]ud\mu\\
&=-  2 t\int_X \left[ \left\|\nabla^2 \log u+{1\over 2}\left({1\over t}-K\right){\bf g}\right\|^2_{\rm HS}+\left({1\over 2}{\partial {\bf g}_t\over \partial t}+\operatorname{\bf Ric_{N, n}}(L)-K{\bf g}\right)(\nabla \log u, \nabla \log u)\right] u d \mu\\
& \hskip0.5cm -\frac{2 t}{N-n} \int_X\left[({\rm Tr}\nabla^2-L)\log u-\frac{N-n}{2} \left(\frac{1}{ t}-K\right)\right]^2 u d \mu.
\end{aligned}
\end{equation*}
This completes the proof of Theorem \ref{mmW1} and Theorem \ref{mmW2}. \hfill $\square$

\section{Shannon entropy power on super Ricci flows on mm spaces}

We now prove the $(-2K)$-concavity of the Shannon entropy power 
 on closed $(K, n, N)$ super Ricci flows.  In the setting of smooth closed $(K, m)$-super Ricci flows or complete Riemannian manifolds with CD$(K, m)$-condition, it has been proved in S. Li-Li \cite{LiLi2024TMJ}.  See 
 \cite{Sh, Costa, Cover} for the background of Shannon entropy power in information theory. 

\begin{theorem}\label{EDIKM}
	Let $(X, d_t, g_t, m_t, \phi_t, t\in [0, T])$ be a family of $n$-dimensional  
 closed metric measure spaces with time dependent metrics and potentials satisfying the conjugate equation \eqref{conjugate}.  Let $u$ be a solution to the heat equation $\partial_t u=L u$, and $H(u)=-\int_X u \log u d\mu$ be the Shannon entropy.  Then 

\begin{align*}
	\frac{1}{2}H''+\frac{{H'}^2}{N} &=
-\frac{1}{N} \int_X \Big[L \log u-\int_Xu  L\log u d\mu \Big]^2ud\mu-\int_X 
\left({1\over 2}{\partial g\over \partial t}+\Ric_{N, n}(L)\right)(\nabla \log u, \nabla \log u) ud\mu\\
&\hskip1cm -\int_X \left\|\nabla^2 \log u-\frac{\Delta\log u} {n}g\right\|^2_{\rm HS}ud\mu-{N-n\over Nn}
\int_X \Big[L\log u+{N\over N-n}\nabla \phi\cdot\nabla \log u\Big]^2ud\mu.
\end{align*}
As a consequence, on every closed $(K, n, N)$-super Ricci flow, the following Riccatti  entropy differential inequality holds
	\begin{equation}\label{EDI}
		H''+\frac{2}{N}H'^2+2KH'\leq 0,
	\end{equation}
	Equivalently, 
the Shannon entropy power $\mathcal{N}(u)=e^{2H(u)\over N}$ is $(-2K)$-concave on every closed 
$(K, n, N)$ super Ricci flow, i.e.,  
\begin{equation*}
{d^2\mathcal{N} \over dt^2} \leq -2K{d\mathcal{N}\over dt}.		
	\end{equation*}
	In particular,  when $K=0$, we have
	\begin{equation*}
{d^2\mathcal{N} \over dt^2} \leq 0.		
	\end{equation*}
	Equivalently, the Shannon entropy power $\mathcal{N}(u)=e^{2H(u)\over N}$ is concave on every closed 
$(0, n, N)$ super Ricci flow. 
\end{theorem}

 \begin{proof} The proof is similar to the case of smooth $(K, m)$-super Ricci flow given by S. Li-Li \cite{LiLi2024TMJ}. Indeed, by \eqref{1H} and \eqref{2H} in Theorem \ref{Th-B1}, we have
 \begin{align*}
	-\frac{1}{2}H''&=\int_X \Gamma_2(\nabla \log u,\nabla \log u )ud\mu\\
	&=\int_X\Big[\|\nabla^2 \log u \|^2_{\operatorname{HS}}+\left({1\over 2}{\partial {\bf g}\over \partial t}+
	\Ric(L)\right)(\nabla \log u,\nabla \log u)\Big] ud\mu\\
	&=\int_X\Big[\frac{(L\log u)^2}{N}+\left({1\over 2}{\partial {\bf g}\over \partial t}+\Ric_{N, n}(L)\right)(\nabla \log u,\nabla \log u)+\left\|\nabla^2 \log u-\frac{\Delta\log u}{n}{\bf g}\right\|^2_{\operatorname{HS}}\Big]ud\mu\\
	& \hskip2cm +{N-n\over Nn}\int_X \Big[L\log u+{N\over N-n}\nabla \phi\cdot\nabla \log u\Big]^2ud\mu\\
	&=\frac{1}{N}\Big(\int_X|\nabla \log u |^2 ud\mu\Big)^2 +\frac{1}{N} \int_X \Big[L \log u-\int_Mu  L\log u d\mu \Big]^2ud\mu\\
	&\hskip1cm +\int_X\Big[\left({1\over 2}{\partial {\bf g}\over \partial t}+\Ric_{N, n}(L)\right)(\nabla \log u,\nabla \log u)+\left\|\nabla^2 \log u-\frac{\Delta\log u}{n}{\bf g}\right\|^2_{\operatorname{HS}}\Big]ud\mu\\
	& \hskip2cm +{N-n\over Nn}\int_X\Big[L\log u+{N\over N-n}\nabla \phi\cdot\nabla \log u\Big]^2ud\mu.
\end{align*}
This yields
\begin{align*}
	\frac{1}{2}H''+\frac{{H'}^2}{N} &=
-\frac{1}{N} \int_X \Big[L \log u-\int_Xu  L\log u d\mu \Big]^2ud\mu-\int_X 
\left({1\over 2}{\partial {\bf g}\over \partial t}+\Ric_{N, n}(L)\right)(\nabla \log u, \nabla \log u) ud\mu\\
&\hskip1cm -\int_X \left\|\nabla^2 \log u-\frac{\Delta\log u} {n}{\bf g}\right\|^2_{\rm HS}ud\mu-{N-n\over Nn}
\int_X \Big[L\log u+{N\over N-n}\nabla \phi\cdot\nabla \log u\Big]^2ud\mu.
\end{align*}
Thus, on every closed $(K, n, N)$ super Ricci flow, we have
\begin{align*}
	\frac{1}{2}H''+\frac{{H'}^2}{N} &\leq -K\int_X|\nabla \log u|^2 ud\mu -\int_X \Big[\left\|\nabla^2 \log u-\frac{\Delta\log u}{n}{\bf g}\right\|^2_{\operatorname{HS}}\Big]ud\mu\\
	& \ \ \ -\frac{1}{N} \int_X\Big[L \log u-\int_X L\log u ud\mu \Big]^2ud\mu.
\end{align*}
The Ricatti EDI reads 
\begin{align*}
	\frac{1}{2}H''+\frac{{H'}^2}{N} +KH'&\leq  -\int_X \Big[\left\|\nabla^2 \log u-\frac{\Delta\log u}{n}{\bf g}\right\|^2_{\operatorname{HS}}\Big]ud\mu-\frac{1}{N} \int_X\Big[L \log u-\int_X L\log u ud\mu \Big]^2ud\mu.
\end{align*}
\quad\quad
In particular, we have
\begin{align*}
	\frac{1}{2}H''+\frac{H'^2}{N}+KH'\leq -{1\over N}\int_X \left|\Delta \log u-\int_X \Delta\log u ud
	\mu\right|^2 ud\mu,\label{BNK6}
\end{align*}
which implies the Riccatti EDI $(\ref{EDI})$ in Theorem \ref{EDIKM}. 
\end{proof}

We can also derive an upper bound for the Fisher information on closed $(K, n, N)$ super Ricci flows.
 
\begin{theorem}\label{Ibound}On every closed $(K, n, N)$ super Ricci flow, we have
\begin{equation}
	I(u(t))=\frac{d}{dt}H(u(t))\leq \frac{NK}{e^{2Kt}-1}. \label{IKbound}
\end{equation}
In particular, on every closed $(0, n, N)$ super Ricci flow, we have
\begin{equation}
	I(u(t))=\frac{d}{dt}H(u(t))\leq \frac{N}{2t}. \label{IKbound0}
\end{equation} 
\end{theorem}
\begin{proof} Based on the Riccatti Entropy Differential Inequality $(\ref{EDI})$,  the proof of Theorem \ref{Ibound} has been essentially given by S. Li-Li \cite{LiLi2024TMJ}. To save the length of the paper, we omit the detail.
\end{proof}

%
%

Closely related to the above Riccatti entropy differential inequality \eqref{EDI}, 
we have the following result which extends the logarithmic  entropy formula (see Ye \cite{Ye} and Wu \cite{Wu}) to $(K, n, N)$ super Ricci flows on mm spaces. 

\begin{theorem}\label{log}
	Let $(X, d, \mu)$ be a closed
$(K, n, N)$ super Ricci flow on mm spaces and $u$ be a positive solution to the heat equation $\partial_t u=L u$. Assume that $a$ is a constant such that $\frac{1}{4}\int_X \frac{|\nabla u |_w^2}{u}d\mu +a>0$. 
Define the logarithmic entropy $\mathcal{Y}_a(u,t)$ as follows
\begin{equation*}
	\mathcal{Y}_a(u,t):=-\int_X u \log u d\mu +\frac{N}{2}\log\Big(\frac{1}{4}\int_X \frac{|\nabla u |_w^2}{u} +a\Big)+(NK-4a)t.
\end{equation*}
Then, we have
	\begin{equation*}
		\frac{d\mathcal{Y}_a}{dt}\leq - \frac{1}{4\omega}\int_X \Big[L\log u+4\omega\Big]^2 u d\mu+{aNK\over w},
	\end{equation*}
	where $\omega=\frac{1}{4}\int_X \frac{|\nabla u |_w^2}{u}d\mu +a$.
	In particular, when $K=0$, it holds
	\begin{equation*}
		\frac{d\mathcal{Y}_a}{dt}\leq - \frac{1}{4\omega}\int_X \Big[L\log u+4\omega\Big]^2 u d\mu\leq 0.
	\end{equation*}
\end{theorem}

We give two proofs of Theorem \ref{log}.  The first one follows the same argument as used in the proof of Theorem 5.1 in \cite{LiLi2024TMJ}, where S. Li and the author proved the Shannon entropy power formula on complete Riemannian manifolds with CD$(K, N)$ condition. The second one uses the same argument as used in Ye \cite{Ye} and Wu \cite{Wu}.  To save the length of the paper, we omit the second one. See Li-Zhang \cite{Li-Zhang2025} for the second proof on RCD$(K, n, N)$ spaces.  

	We first prove the following result on RCD$(K, n, N)$ space, which was first proved by S.Li and the first named author \cite{LiLi2024TMJ} on complete Riemannian manifolds with bounded geometry condition.

\begin{theorem}\label{HRiccati} On every  metric measure space with time dependent metrics and potentials satisfying the conjugate equation \eqref{conjugate}, we have

\begin{equation}\label{LLHK}
\begin{split}
&H''+{2\over N}H'^2+2KH'\\
&=
-2\int_X \left[\left\|\nabla^2\log u-{{\rm Tr}\nabla^2\log u\over n}{\bf g}\right\|_{\rm HS}^2+\left({1\over 2}{\partial {\bf g}\over \partial t}+{\bf Ric_{N, n}}(L)-K{\bf g}\right)(\nabla \log u, \nabla \log u)\right] ud\mu\\
& \hskip0.5cm  -{2(N-n)\over Nn}\int_X \left[L \log u+{N\over N-n}({\rm Tr}\nabla^2-L)\log  u\right]^2ud\mu -{2\over N}\int_X \left[L \log u-\int_X L \log u ud\mu\right]^2 ud\mu.
	\end{split}
\end{equation}
In particular, on any $(K, n, N)$ super Ricci flow we have

\begin{equation*}
\begin{split}
H''+{2\over N}H'^2+2KH'\leq -{2\over N}\int_X \left[L \log u-\int_X L \log u ud\mu\right]^2 ud\mu.
	\end{split}
\end{equation*}

\end{theorem}
\begin{proof}

The proof is as the same as in \cite{LiLi2024TMJ}.See also Li-Zhang \cite{Li-Zhang2025}. By the second order entropy dissipation formula \eqref{2H}, we have

\begin{equation}\label{HDF2nd}
\begin{split}
	H''
	&=-2\int_X \left[\left\|\nabla^2\log u-{{\rm Tr}\nabla^2\log u\over n}{\bf g}\right\|_{\rm HS}^2+
	\left({1\over 2}{\partial {\bf g}\over \partial t}+{\bf Ric_{N, n}}(L)-K{\bf g}\right)(\nabla \log u, \nabla \log u)\right] ud\mu\\
& \ \hskip2cm -2\int_X \left[{|{\rm Tr}\nabla^2\log u|^2\over n}+{|({\rm Tr}\nabla^2-L)\log u|^2\over N-n}\right] 
ud\mu.
	\end{split}
\end{equation}
Using 
$$(a+b)^2={a^2\over 1+\varepsilon}-{b^2\over \varepsilon}+{\varepsilon\over 1+\varepsilon}\left(a+{1+\varepsilon\over \varepsilon}b\right)^2$$
and taking $a=L \log u$, $b={\rm Tr}\nabla^2\log u-L \log u$ and $\varepsilon={N-n\over n}$, we have\begin{equation*}
\begin{split}
{|{\rm Tr}\nabla^2 \log u|^2\over n}&={|L \log u+({\rm Tr}\nabla^2-L)\log  u|^2\over n}\\
&={|L\log u|^2\over N}-{|({\rm Tr}\nabla^2-L)\log  u|^2\over N-n}+{N-n\over Nn}\left[L \log u+{N\over N-n}({\rm Tr}\nabla^2-L)\log  u\right]^2.
\end{split}
\end{equation*}
Substituting it into \eqref{HDF2nd}, we  can derive \eqref{LLHK}. 

%
%
%

\end{proof}

\noindent{\bf Proof of Theorem \ref{log} using \eqref{LLHK} in Theorem \ref{HRiccati}}. 
By the definition of $\mathcal{Y}_a(u, t)$
\begin{equation*}
	\mathcal{Y}_a(u,t):=H (u(t))+\frac{N}{2}\log\Big(\frac{1}{4}{d\over dt}H(u(t)) +a\Big)+(NK-4a)t,
\end{equation*}
we have
\begin{equation*}
\begin{split}
	{d\over dt}\mathcal{Y}_a(u,t)
	&={N\over 2(H'+4a)}\left[H''+{2\over N}H'^2+2KH'+{8a(NK-4a)\over N}\right].
\end{split}
\end{equation*}
%
Using 

\begin{align*}
	\int_X \Big[L \log u-\int_Xu  L\log u d\mu \Big]^2ud\mu+16a^2
	=\int_X \Big[L \log u-\int_Xu  L\log u d\mu +4a\Big]^2ud\mu,
\end{align*}
we have
\begin{align*}
	{d\over dt}\mathcal{Y}_a(u, t)
&= -\frac{1}{4w} \int_X \Big[L \log u+4w \Big]^2ud\mu-{N\over 4w}\int_X 
\left({1\over 2}{\partial {\bf g}\over \partial t}+{\bf Ric_{N, n}}(L)-K{\bf g}\right) (\nabla \log u, \nabla \log u) ud\mu\\
&\hskip1cm -{N\over 4w}\int_X \left\|\nabla^2 \log u-\frac{{\rm Tr}\nabla^2\log u} {n}{\bf g}\right\|^2_{\rm HS}ud\mu\\
&\hskip1.5cm-{N-n\over 4Nnw}
\int_X \Big[L\log u+{N\over N-n}({\rm Tr}\nabla^2-L) \log u\Big]^2ud\mu+{aNK\over w}.
\end{align*}
Thus, on any $(K, n, N)$ super Ricci flow, we have
\begin{align*}
	{d\over dt}\mathcal{Y}_a(u, t)\leq
	-\frac{1}{4w} \int_X \Big[L\log u+4w \Big]^2ud\mu+{aNK\over w}.
	\end{align*}
	In particular, on any $(0, n, N)$ super Ricci flow, we have
	\begin{align*}
	{d\over dt}\mathcal{Y}_a(u, t)\leq
	-\frac{1}{4w} \int_X \Big[L\log u+4w \Big]^2ud\mu\leq 0.
	\end{align*}
This finishes the proof of Theorem \ref{log}. \hfill $\square$

\section{The Li-Yau-Hamilton-Perelman Harnack inequality}

In this section, inspired by Perelman's seminal work \cite{P1}, we prove the Li-Yau-Hamilton-Perelman Harnack inequality on super Ricci flows. 

\subsection{LYHP Harnack inequalities on Ricci flows and Riemannian manifolds}

In \cite{P1}, Perelman introduced the quantity 
\begin{eqnarray}
\nu= [\tau (2\Delta f-|\nabla f|^2+R)+f-n]{e^{-f}\over (4\pi t)^{n/2}}. \label{LYHP1}
 \end{eqnarray}
and proved that the $W$-entropy is naturally related to the quantity $\nu$ as follows 
 \begin{eqnarray}
W(g, f, \tau)&=&\int_M \nu dv, \label{LYHP-W1}
\end{eqnarray}
and the $W$-entropy derivation formula can be  reformulated as follows
\begin{eqnarray}
{d\over dt}W(g, f, \tau)&=&\int_M \square^* \nu dv,  \label{LYHP-W2}
\end{eqnarray}
where
\begin{eqnarray}
\square^*=-{\partial\over \partial t}-\Delta+R.  \label{LYHP-W3}
\end{eqnarray} 
Moreover, Perelman proved the following Li-Yau-Hamilton  Harnack inequality
for the fundamental solution of the conjugate backward heat equation of the Ricci flow.

\begin{theorem} [Perelman \cite{P1}]
Let $g(t)$ be the Ricci flow on $M\times (0, T)$, i.e., $$\partial_t g=-2Ric.$$ Let 
$H={e^{-f}\over (4\pi t)^{n/2}}$ 
 be the fundamental solution to the conjugate  heat equation
\begin{eqnarray*}
\partial_t u = -\Delta u-Ru.
\end{eqnarray*}
Let
\begin{eqnarray}
\nu_H= [\tau(2\Delta f-|\nabla f|^2+R)+f-n]H. \label{LYHP1}
\end{eqnarray}
Then 
\begin{eqnarray}
\nu_H\leq 0. \label{LYHP2}
\end{eqnarray}
Moreover
\begin{eqnarray*}
\square^*\nu_H=-2\tau \left\|Ric+{\rm Hess} f-{g\over 2\tau}\right\|_{\rm HS}^2H.
\end{eqnarray*}

\end{theorem}

In the sequel of this paper, we call $(\ref{LYHP1})$ the Li-Yau-Hamilton-Perelman  Harnack quantity, and we call $(\ref{LYHP2})$ the Li-Yau-Hamilton-Perelman Harnack inequality for the Ricci flow. 

In \cite{N1, N2}, Ni proved the  Li-Yau-Hamilton-Perelman type Harnack inequality
for the fundamental solution of the heat equation on closed Riemannian manifolds with non-negative Ricci curvature.

\begin{theorem} [Ni \cite{N1, N2}]
Let $(M, g)$ be a closed  Riemannian manifold with $Ric\geq 0$, 
$$H={e^{-f}\over (4\pi t)^{n/2}}$$ 
  the fundamental solution to the heat equation
\begin{eqnarray*}
\partial_t u = \Delta u.
\end{eqnarray*}
Let
\begin{eqnarray*}
\nu_H= [t(2\Delta f-|\nabla f|^2)+f-n]H.
\end{eqnarray*}
Then
\begin{eqnarray*}
\nu_H\leq 0.
\end{eqnarray*}
Moreover,
\begin{eqnarray*}
\left({\partial\over \partial t}-\Delta\right)\nu_H=-2t\left[\left\|\nabla^2 f-{g\over 2t}\right\|_{\rm HS}^2+
Ric(\nabla f, \nabla f)\right] H. 
\end{eqnarray*}

\end{theorem}

In \cite{LiXu}, J. Li and X. Xu extended Ni's result to closed Riemannian manifolds with Ricci curvature bounded from below. More precisely, they proved the following

\begin{theorem} [Li-Xu \cite{LiXu}]
Let $(M, g)$ be a closed Riemannian manifold with $Ric\geq -K$,  where $K\geq 0$ is a constant. Let 
$$H={e^{-f}\over (4\pi t)^{n/2}}$$ be 
  the fundamental solution to the heat equation
\begin{eqnarray*}
\partial_t u = \Delta u.
\end{eqnarray*}
Define
\begin{eqnarray}
\nu_H=\left[t\Delta f+t(1+Kt)(\Delta f -|\nabla f|^2)+f-n\left(1+{1\over 2}Kt\right)^2\right]H.
\end{eqnarray}
Then
\begin{eqnarray*}
\nu_H\leq 0.
\end{eqnarray*}
Moreover,
\begin{eqnarray*}
\left({\partial\over \partial t}-\Delta\right)\nu_H=-2t\left[\left\|\nabla^2 f-{1\over 2}\left({1\over t}+K\right)g
\right\|_{\rm HS}^2+(Ric+Kg)(\nabla f, \nabla f)\right] H. 
\end{eqnarray*}

\end{theorem}

It is natural to ask the question whether we can establish the Li-Yau-Hamilton-Perelman differential Harnack inequality for the Witten Laplacian on compact Riemannian manifolds equipped with 
weighted volume measure and on
closed manifolds with super Ricci flows. 
The purpose of this section is to study  this question.

\subsection{LYHP Harnack inequality on weighted complete Riemannian manifolds}

The following result  is a modified version of the Li-Yau-Hamilton-Perelman differential Harnack 
inequality for the heat equation of the Witten Laplacian on complete Riemannian manifolds with the  $(0, m)$-condition. It was proved in our 2007 unpublished manuscript. 

\begin{theorem} \label{LYHPW1}
Let $(M, g, \phi)$ be a complete Riemannian manifold with $Ric_{m, n}(L)\geq 0$,  $P_t=e^{tL}$ be the heat semigroup generated by $L$, 
and  $H={e^{-f}\over (4\pi t)^{m/2}}$  the fundamental solution to the heat equation $\partial_t u = L u$.
Let
\begin{eqnarray*}
\nu_H(t)= [t(2L f-|\nabla f|^2)+f-m]H.
\end{eqnarray*}
Then
\begin{eqnarray*}
{d\over dt}(P_{T-t}\nu_H(t))\leq 0.
\end{eqnarray*}
Moreover,
\begin{eqnarray*}
W_m(u, t)=\int_M \nu_H d\mu,
\end{eqnarray*}
and
\begin{eqnarray*}\\
{d\over dt}W_m(u, t)=\int_M 
\left({\partial\over \partial t}-L\right)\nu_H d\mu.
\end{eqnarray*}
\end{theorem}

To prove the LYHP Harnack inequality, we need the following 

\begin{lemma}\label{lemma1}  Let $u$ be a positive solution to the heat equation $\partial_t u=Lu$, $f=-\log u$ and $w=2Lf-|\nabla f|^2$. Then
\begin{eqnarray}
\left({\partial\over \partial t}-L\right)w~~&=&-2|\nabla^2 f|^2-2Ric(L)(\nabla f, \nabla f)-2\langle \nabla w, \nabla f\rangle\nonumber\\
&=&-2\Gamma_2(f, f)-2\langle \nabla w, \nabla f\rangle,\label{formula1}.
\end{eqnarray}
When $\phi=0$, $m=n$ and $L=\Delta$, this is due to Ni \cite{N1, N2}.
\end{lemma}
{\it Proof}.  Note that $f_t=Lf-|\nabla f|^2$ and $w=2f_t+|\nabla f|^2$. 
Using  the generalized Bochner formula, a direct calculation yields
\begin{eqnarray*}
(\partial_t-L)w&=&(\partial_t-L)(2f_t+|\nabla f|^2)\\
&=&2(\partial_t-L)f_t+\partial_t |\nabla f|^2-L|\nabla f|^2\\ 
&=&2\partial_t (\partial_t-L)f+\partial_t|\nabla f|^2-L|\nabla f|^2\\
&=&-2\partial_t |\nabla f|^2+\partial_t |\nabla f|^2 -L|\nabla f|^2\\
&=&-\partial_t |\nabla f|^2-L|\nabla f|^2\\
&=&-2\langle \nabla f, \nabla f_t\rangle -2\langle \nabla f, \nabla Lf\rangle  -2|\nabla^2 f|^2-2Ric(L)(\nabla f, \nabla f)\\
&=&-2\langle \nabla f, \nabla (f_t+Lf)\rangle -2|\nabla^2 f|^2-2Ric(L)(\nabla f, \nabla f)\\
&=&-2\langle \nabla f, \nabla w\rangle -2\Gamma_2(f, f).
\end{eqnarray*}
This proves $(\ref{formula1})$.  \hfill $\square$

\begin{lemma}\label{lemma2}  Let $u={e^{-f}\over (4\pi t)^{m/2}}$, $w=2Lf-|\nabla f|^2$, $w_m=tw+f-m$, and $\nu_H=w_mH$. Then
\begin{eqnarray}
\left({\partial\over \partial t}-L\right)w_m&=&-2t \left[\left\|\nabla^2 f-{g\over 2t}\right\|_{\rm HS}^2+Ric_{m, n}(L)(\nabla f, \nabla f)\right] \nonumber\\
& & -{2t\over m-n}\left(\nabla \phi\cdot\nabla f+{m-n\over 2t}\right)^2-2\langle \nabla w_m, \nabla f\rangle,\label{formula2}\\
\left({\partial\over \partial t}-L\right)\nu_H&=&-2t\left[\left\|\nabla^2 f-{g\over 2}\right\|_{\rm HS}^2+Ric_{m, n}(L)(\nabla f, \nabla f)\right] H\nonumber \\
& &\ \ \ -{2t\over m-n}\left(\nabla \phi\cdot\nabla f+{m-n\over 2t}\right)^2 H.\label{formula3}
\end{eqnarray}
When $\phi=0$, $m=n$ and $L=\Delta$, this is due to Ni \cite{N1, N2}.
\end{lemma}
{\it Proof}.  Let $\bar{f}=-\log u$. Then $f=\bar{f}-{m\over 2}\log (4\pi t)$, $\nabla f=\nabla \bar{f}$, $Lf=L\bar{f}$ and $\Gamma_2(f, f)=\Gamma_2(\bar{f}, \bar{f})$. Hence
\begin{eqnarray*}
w_m=tw+\bar{f}-{m\over 2}\log(4\pi t)-m.
\end{eqnarray*}  By the fact  $(\partial_t-L)\bar{f}=-|\nabla \bar{f}|^2$ and  using Lemma \ref{lemma1}, we have
\begin{eqnarray*}
(\partial_t-L)w_m&=&w+t(\partial_t-L)w+(\partial_t-L)(\bar{f}-{m\over 2}\log(4\pi t))\\
&=& 2\bar{f}_t+|\nabla \bar{f}|^2-2t \langle \nabla \bar{f}, \nabla w\rangle -2t \Gamma_2(\bar{f}, \bar{f})-|\nabla \bar{f}|^2-{m\over 2t}\\
&=& 2\bar{f}_t-2t \langle \nabla f, \nabla w\rangle -2t \Gamma_2(f, f)-{m\over 2t}.
\end{eqnarray*}
Now 
\begin{eqnarray*}
\langle \nabla w_m, \nabla f\rangle =t 
\langle \nabla w, \nabla f\rangle +|\nabla f|^2.
\end{eqnarray*}
Thus
\begin{eqnarray*}
(\partial_t-L)w_m
&=& 2Lf-2|\nabla f|^2-2\langle \nabla f, \nabla w_m\rangle +2 |\nabla f|^2 -2t \Gamma_2(f, f)-{m\over 2t}\\
&=&2Lf-2\langle \nabla f, \nabla w_m\rangle -2t \Gamma_2(f, f)-{m\over 2t}.
\end{eqnarray*}
Note that
\begin{eqnarray*}
& &2Lf -2t \Gamma_2(f, f)-{m\over 2t}\\
&=&2\Delta f-2\langle \nabla \phi, \nabla f\rangle -2t |\nabla^2 f|^2-2t Ric(L)(\nabla f, \nabla f)-{m\over 2t}\\
&=&-2t\left[\left\|\nabla^2 f-{g\over 2t}\right\|_{\rm HS}^2+Ric_{m, n}(L)(\nabla f, \nabla f)\right]-{2t\over m-n} \left(\nabla \phi\cdot \nabla f+{m-n\over 2t}  \right)^2.
\end{eqnarray*}
This  proves $(\ref{formula2})$.  Using the fact that $L(w_mH)=Lw_m H+w_m LH+2\langle \nabla w_m, \nabla H\rangle$, and $\nabla H=-H \nabla f$,  we can derive $(\ref{formula3})$ from  $(\ref{formula2})$. The proof of Lemma \ref{lemma1} is completed. \hfill $\square$

\medskip

\noindent{\bf Proof of Theorem \ref{LYHPW1}.}  We use the same argument as Perelman \cite{P1} for the proof of the LYHP inequality for the conjugate heat equation for Ricci flow. Let $P_t=e^{tL}$ be the heat semigroup generated by $L$. Then $h(t)=P_{T-t}h(T)$ is the unique solution of the backward heat equation 
$$\partial_t h=-Lh$$ with terminal data $h(T)>0$.  Taking the time derivative, we have
\begin{eqnarray*}
{d\over dt}\int_M h(t)\nu_H(t)d\mu&=&\int_M \partial_t h \nu_H d\mu+\int_M h \partial_t \nu_H d\mu\\
 &=& -\int_M Lh\nu_H d\mu+\int_M \partial_t\nu_H h d\mu\\
 &=& -\int_M L\nu_H h d\mu+\int_M \partial_t\nu_H h d\mu\\
&=& \int_M (\partial_t-L)\nu_H hd\mu, 
\end{eqnarray*}
where in the third step we have used the fact that $L$ is self-adjoint with respect to $\mu$. By Lemma  \ref{lemma2}, we have 
\begin{eqnarray*}
(\partial_t-L)\nu_H\leq 0,
\end{eqnarray*}
which yields 
\begin{eqnarray*}
{d\over dt}\int_M h(t)\nu_H(t)d\mu\leq 0.
\end{eqnarray*}
Writing $h(t)=P_{T-t}h(T)$ and using integration by parts, we have
\begin{eqnarray*}
{d\over dt}\int_M h(T)P_{T-t}\nu_H(t)d\mu\leq 0.
\end{eqnarray*}
As $h(T)$ can be arbitrary positive function, this yields
\begin{eqnarray*}
{d\over dt}(P_{T-t}\nu_H(t))\leq 0. 
\end{eqnarray*}
The proof of Theorem \ref{LYHPW1} is finished. \hfill $\square$

The following result is a natural extension of Li-Xu's  LYHP Harnack inequality on weighted Riemannian manifolds with the $CD(-K, m)$ condition.

\begin{theorem} \label{LYHPW2}
Let $(M, g)$ be a closed Riemannian manifold, $\phi\in C^2(M)$. Suppose that the $CD(-K, m)$ condition holds, i.e., $Ric_{m, n}(L)\geq -K$,  where $K\geq 0$ is a constant. Let 
$$H={e^{-f}\over (4\pi t)^{m/2}}$$ be 
  the fundamental solution to the heat equation of the Witten Laplacian 
\begin{eqnarray*}
\partial_t u = L u.
\end{eqnarray*}
Define
\begin{eqnarray}
\nu_H=\left[tL f+t(1+Kt)(L f -|\nabla f|^2)+f-m\left(1+{1\over 2}Kt\right)^2\right]H.
\end{eqnarray}
Then
\begin{eqnarray*}
{d\over dt}(P_{T-t}\nu_H(t))\leq 0.
\end{eqnarray*}
Moreover,
\begin{eqnarray*}
W_m(u, t)=\int_M \nu_H d\mu,
\end{eqnarray*}
and
\begin{eqnarray*}
{d\over dt}W_m(u, t)=\int_M 
\left({\partial\over \partial t}-L\right)\nu_H d\mu.
\end{eqnarray*}
\end{theorem}
{\it Proof}. The proof is analogue to the  one of Theorem \ref{LYHPW1}. \hfill $\square$

\medskip
We would like to give the following 

\begin{remark} \label{Wuhu2008}Can we prove that
$\lim\limits_{t\rightarrow 0} \nu_H(t)\leq 0$? 
This needs the Gaussian heat kernel lower bound estimate for $L$ and the volume equivalence  $\mu(B(x, r))\simeq  C_n r^n$ for $r\rightarrow 0+$. It is true for $L=\Delta$, $\mu=v$, $m=n$ with $Ric\geq 0$ by using the Cheeger-Yau Gaussian lower bound heat kernel estimate \cite{CY}. See \cite{N1, N2}. However,  the Cheeger-Yau Gaussian lower bound heat kernel estimate  is not true in general for $L\neq \Delta$ and $d\mu=e^{-\phi}dv$ even $Ric_{m, n}(L)\geq 0$ for $m>n$.  See \cite{Li2012}. For this reason, we have not submitted the results in this subsection obtained in 2007 until now. See also the slides of author's talk \cite{Li2008Wuhu} entitled  ``Differential Harnack inequality and Perelman's entropy formula on complete Riemannian manifolds'' in 2008 Workshop on Markov Processes and Related Fields organized by Prof. Mufa Chen in Wuhu. 
\end{remark}

\subsection{LYHP Harnack inequality on smooth super Ricci flows}

In the case where $(M, g(t), \phi(t))$ is a manifold with time-dependent metric such that $d\mu=e^{-\phi(t)}dvol_{g(t)}$ does not change, we have the following lemma which was proved in S. Li and the author's papers \cite{LiLi2018SCM, LiLi2019SCM} for the proof of the Li-Yau Harnack type estimate on smooth super Ricci flows.

\begin{lemma} [Li-Li \cite{LiLi2018SCM, LiLi2019SCM}]\label{lemma3} 
Let $M$ be a  manifold with a family of time dependent metrics $(g(t), t\in [0, T])$ and potentials $\phi(t)\in C^2(M)$, $t\in [0, T]$. Suppose that $d\mu=e^{-\phi(t)}dvol_{g(t)}$ does not change, i.e., the conjugate heat equation \eqref{conjugate} holds. Let $\partial_t g=2h$.
 For any $f\in C^\infty(M)$, we have
\begin{eqnarray*}
\partial_t |\nabla f|^2=-{\partial g\over \partial t}(\nabla f, \nabla f)+2\langle \nabla f, \nabla f_t\rangle,
\end{eqnarray*}
and
\begin{eqnarray}
[\partial_t, L] f=-2\langle h, \nabla^2 f\rangle+2h(\nabla \phi, \nabla f)-\langle 2 {\rm div}h-\nabla {\rm Tr}_g h+\nabla \partial_t \phi, \nabla f\rangle. \label{commu}
\end{eqnarray}
\end{lemma}
{\it Proof}. For the completeness of the paper, we allow us to reproduce the proof here. By direct calculation, we have
\begin{eqnarray*}
\partial_t |\nabla f|^2=\partial_t g^{ij}(t)\nabla_i f\nabla_j f=\partial_t g^{ij}(t)\nabla_i f\nabla_j f+2g^{ij}(t)\nabla_i f\nabla_j f_t.
\end{eqnarray*}
Note that
\begin{eqnarray*}
\partial_t g^{ij}(t)=-\partial_t g_{ij}(t)=-2h_{ij}.
\end{eqnarray*}
The first equality follows. On the other hand, by \cite{CLN}, we have
\begin{eqnarray*}
\partial_t \Delta_{g(t)} f=\Delta_{g(t)}\partial_t f-2\langle h, \nabla^2 f\rangle-2\langle {\rm div}h- {1\over 2}\nabla{\rm Tr}_g h, \nabla f\rangle.
\end{eqnarray*}
Combining this with
\begin{eqnarray*}
\partial_t \langle \nabla \phi, \nabla f\rangle=-\partial_t g(\nabla \phi, \nabla f)+\langle \nabla\phi_t, \nabla f\rangle+\langle \nabla \phi, \nabla f_t\rangle,
\end{eqnarray*}
we obtain  $(\ref{commu})$ in Lemma \ref{lemma3}.  \hfill $\square$

\begin{eqnarray*}
\partial_t L f&=&\partial_t \Delta_{g(t)} f-\partial_t \langle \nabla \phi, \nabla f\rangle\\
&=&\Delta_{g(t)} \partial_t f-2\langle h, \nabla^2 f\rangle-2\langle {\rm div}h- {1\over 2}\nabla{\rm Tr}_g h, \nabla f\rangle\\
& &\hskip1cm +2h(\nabla \phi, \nabla f)-\langle \nabla\phi_t, \nabla f\rangle-\langle \nabla \phi, \nabla f_t\rangle\\
&=&L\partial_t f-2\langle h, \nabla^2 f\rangle+2h(\nabla \phi, \nabla f)-\langle 2{\rm div}h-\nabla {\rm Tr}_g h+\nabla \phi_t, \nabla f\rangle.
\end{eqnarray*}
This finishes the proof of Lemma \ref{lemma3}. \hfill $\square$

\begin{lemma}\label{lemma4}  Let $u$ be a positive solution to the backward heat equation $\partial_t u=-Lu$, $f=-\log u$ and $w=2Lf-|\nabla f|^2$. Then
\begin{eqnarray}
\left({\partial\over \partial t}-L\right)w~=\left({\partial g\over \partial t}-2\Gamma_2\right)(\nabla f, \nabla f)-2\langle \nabla w, \nabla f\rangle+2[\partial_t, L]f,\label{formula5}.
\end{eqnarray}
\end{lemma}
{\it Proof}.  Note that $f_t=Lf-|\nabla f|^2$ and $w=2f_t+|\nabla f|^2$. 
Using  the generalized Bochner formula, a direct calculation yields
\begin{eqnarray*}
(\partial_t-L)w&=&(\partial_t-L)(2f_t+|\nabla f|^2)\\
&=&2\partial_t^2 f- 2L\partial_t f+\partial_t |\nabla f|^2-L|\nabla f|^2\\ 
&=&2\partial_t (\partial_t-L)f+2[\partial_t, L]f+\partial_t|\nabla f|^2-L|\nabla f|^2\\
&=&-2\partial_t |\nabla f|^2+2[\partial_t, L]f+\partial_t |\nabla f|^2 -L|\nabla f|^2\\
&=&-\partial_t |\nabla f|^2-L|\nabla f|^2+2[\partial_t, L]f\\
&=&{\partial g\over \partial t}(\nabla f, \nabla f)-2\langle \nabla f, \nabla f_t\rangle -2\langle \nabla f, \nabla Lf\rangle  -2\Gamma_2(\nabla f, \nabla f)+2[\partial_t, L]f\\
&=&-2\langle \nabla f, \nabla (f_t+Lf)\rangle +\left({\partial g\over \partial t}-2\Gamma_2\right)(\nabla f, \nabla f)+2[\partial_t, L]f\\
&=&-2\langle \nabla f, \nabla w\rangle+\left({\partial g\over \partial t}-2\Gamma_2\right)(\nabla f, \nabla f)+2[\partial_t, L]f.
\end{eqnarray*}
This proves $(\ref{formula5})$.  \hfill $\square$
%
%
%
%

\begin{lemma}\label{lemma6} Let $\tau=T-t$, and $H={e^{-f}\over (4\pi \tau)^{m/2}}$ be a positive solution to the heat equation $\partial_\tau u=Lu$. Let $w=2L\log H-|\nabla \log H|^2$, $w_m=\tau w+f-m$, $\nu_H=w_mH$. Denote $\square^*=\partial_t-L$. Then
\begin{eqnarray}
\square^* w&=&-2|\nabla^2 f|^2-2\left({1\over 2}{\partial g\over \partial t}+Ric(L)\right)(\nabla f, \nabla f)-2\langle w, \nabla f\rangle+2[\partial_\tau, L]\bar{f},\label{formula1b}\\
\square^* w_m&=&-2\tau \left[\left\|\nabla^2 f-{g\over 2\tau}\right\|_{\rm HS}^2+2\left({1\over 2}{\partial g\over \partial t}+Ric_{m, n}(L)\right)(\nabla f, \nabla f)\right]\nonumber\\
& & -{2\tau\over m-n}\left(\nabla \phi\cdot\nabla f+{m-n\over 2\tau}\right)^2-2\langle \nabla w_m, \nabla f\rangle-2\tau [\partial_\tau, L] \log H, \label{formula2b}\\
\square^*\nu_H&=&-2\tau\left[\left\|\nabla^2 f-{g\over 2\tau}\right\|_{\rm HS}^2+2\left({1\over 2}{\partial g\over \partial t}+Ric_{m, n}(L)\right)(\nabla f, \nabla f)\right]H\nonumber\\
& &\ \ \ -{2\tau\over m-n}\left(\nabla \phi\cdot\nabla f+{m-n\over 2\tau}\right)^2 H-2\tau [\partial_\tau, L]\log H H.\label{formula3b}
\end{eqnarray}

\end{lemma}

{\it Proof}.  Let $\bar{f}=-\log u$. Then $f=\bar{f}-{m\over 2}\log (4\pi \tau)$, $\nabla f=\nabla \bar{f}$, $Lf=L\bar{f}$ and $\Gamma_2(f, f)=\Gamma_2(\bar{f}, \bar{f})$. Hence
\begin{eqnarray*}
w_m=\tau w+\bar{f}-{m\over 2}\log(4\pi \tau)-m.
\end{eqnarray*}  By the fact  $(\partial_\tau-L)\bar{f}=-|\nabla \bar{f}|^2$ and  using Lemma \ref{lemma1}, we have
\begin{eqnarray*}
(\partial_\tau-L)w_m&=&w+\tau(\partial_\tau-L)w+(\partial_\tau-L)(\bar{f}-{m\over 2}\log(4\pi \tau))\\
&=& 2\bar{f}_\tau+|\nabla \bar{f}|^2-2\tau \langle \nabla \bar{f}, \nabla w\rangle -2\tau \Gamma_{2, \tau}(\bar{f}, \bar{f})+2\tau[\partial_\tau, L] \bar{f}-|\nabla \bar{f}|^2-{m\over 2\tau}\\
&=& 2\bar{f}_\tau-2\tau \langle \nabla f, \nabla w\rangle -2\tau \Gamma_{2, \tau}(f, f)+2\tau[\partial_\tau, L]\bar{f}-{m\over 2\tau}.
\end{eqnarray*}
Now 
\begin{eqnarray*}
\langle \nabla w_m, \nabla f\rangle =\tau 
\langle \nabla w, \nabla f\rangle +|\nabla f|^2.
\end{eqnarray*}
Thus
\begin{eqnarray*}
(\partial_\tau-L)w_m
&=& 2Lf-2|\nabla f|^2-2\langle \nabla f, \nabla w_m\rangle +2 |\nabla f|^2 -2\tau \Gamma_{2, \tau}(f, f)+2\tau[\partial_\tau, L]\bar{f}-{m\over 2\tau}\\
&=&2Lf-2\langle \nabla f, \nabla w_m\rangle -2\tau \Gamma_{2, \tau}(f, f)+2\tau [\partial_\tau, L]\bar{f}-{m\over 2t}.
\end{eqnarray*}
Note that
\begin{eqnarray*}
& &2Lf -2\tau \Gamma_{2, \tau}(f, f)-{m\over 2\tau}\\
&=&2\Delta f-2\langle \nabla \phi, \nabla f\rangle -2\tau |\nabla^2 f|^2-2\tau \left({1\over 2}{\partial g\over \partial t}+Ric(L)\right)(\nabla f, \nabla f)-{m\over 2t}\\
&=&-2\tau\left[\left\|\nabla^2 f-{g\over 2\tau}\right\|_{\rm HS}^2+  \left({1\over 2}{\partial g\over \partial t}+Ric_{m, n}(L)\right)(\nabla f, \nabla f)\right]-{2\tau\over m-n} \left(\nabla \phi\cdot \nabla f+
{m-n\over 2\tau}  \right)^2.
\end{eqnarray*}
This  proves $(\ref{formula2b})$.  Using the fact that $L(w_mH)=Lw_m H+w_m LH+2\langle \nabla w_m, \nabla H\rangle$, and $\nabla H=-H \nabla f$, we derive $(\ref{formula3b})$ from  $(\ref{formula2b})$, i.e., 
\begin{eqnarray*}
(\partial_t-L)v_H=W_H+2\tau[\partial_\tau, L]\bar{f}H.
\end{eqnarray*}
 The proof of Lemma \ref{lemma1} is completed. \hfill $\square$

Now we prove the Li-Yau-Hamilton-Perelman differential Harnack inequality on super Ricci flows. 

\begin{theorem}\label{LYHP} \label{LYHPW3} Let $(M, g(t), \phi(t), t\in [0, T])$ be a family of time-dependent closed Riemannian 
manifolds with time dependent metric and potentials satisfying the conjugate heat equation \eqref{conjugate}. Let
\begin{eqnarray*}
\nu_H(t)= [t(2L f-|\nabla f|^2)+f-m]H.
\end{eqnarray*}
Then
\begin{eqnarray*}
{d\over dt}\left(P_{T, t}^*\nu_H(t)\right)=P_{T, t}^*\left(W_H+2\tau [\partial_\tau, L]\log H H\right)\leq 2\tau P_{T, t}^*\left([\partial_\tau, L]\log H H\right),
\end{eqnarray*}
where $P_{T, t}^*$ is the adjoint of the operator $P_{T, t}$ from $L^2((M, g_{T}), \mu)$ to $L^2((M, g_{t}), \mu)$, and
\begin{eqnarray*}
W_H&=&-2t\left[\left\|\nabla^2 f-{g\over 2}\right\|_{\rm HS}^2+2\left({1\over 2}{\partial g\over \partial t}+Ric_{m, n}(L)\right)(\nabla f, \nabla f)\right] H\\
& &\ \ \ \ \ \ -{2t\over m-n}\left(\nabla \phi\cdot\nabla f+{m-n\over 2t}\right)^2H.
\end{eqnarray*}
In particular, if $(M, g, \phi)$ is  time-independent and satisfies the CD$(0, m)$-condition, i.e., 
$Ric_{m, n}(L)\geq 0$, then $[\partial_\tau, L]=0$ and hence 
\begin{eqnarray*}
{d\over dt}(P_{T-t}\nu_H(t))&=&P_{T-t}(W_H)\leq 0.
\end{eqnarray*}  
\end{theorem}
{\it Proof}. Let
$h(t)=P_{T, t}h$ be the solution of the backward heat equation 
$$\partial_t h=-Lh$$ with terminal data $h(T)=h>0$. Then
\begin{eqnarray*}
{d\over dt}\int_M h(t)\nu_H(t)d\mu
&=&\int_M \partial_t h \nu_H d\mu+\int_M h \partial_t \nu_H d\mu\\
 &=&\int_M (\partial_t+L)h\nu_H d\mu+\int_M (\partial_t-L)\nu_H h d\mu\\
&=&\int_M (\partial_t-L)\nu_H hd\mu\\
&=&\int_M \left(W_H +2\tau [\partial_\tau, L]\log H H\right) h d\mu.
\end{eqnarray*}

Writing $h(t)=P_{T, t}h(T)$ and using integration by parts, we have
\begin{eqnarray*}
{d\over dt}\int_M h(T) P_{T, t}^*\nu_H(t)d\mu=\int_M P_{T, t}^*\left(W_H+2\tau [\partial_\tau, L]\log H H\right) h(T)d\mu.
\end{eqnarray*}
Note that, when $M$ is compact, we have
\begin{eqnarray*}
{d\over dt}\int_M h(T) P_{T, t}^*\nu_H(t)d\mu=\int_M h(T) {d\over dt}\left(P_{T, t}^*\nu_H(t)\right) d\mu.
\end{eqnarray*}
As $h(T)$ can be arbitrary positive function, this prove Theorem \ref{LYHPW3}. \hfill $\square$

Now we can  reformulate the $W$-entropy formula on super Ricci flows as follows.

\begin{theorem} \label{WH} Let $(M, g(t), \phi(t), t\in [0, T])$ be a family of time-dependent closed Riemannian 
manifolds with time dependent metric and potentials satisfying the conjugate heat equation \eqref{conjugate}. 
Then
\begin{eqnarray*}
W_{m}(u, t)&=&\int_M \nu_H d\mu,
\end{eqnarray*}
and
\begin{eqnarray*}
{d\over dt}W_{m}(u, t)&=&\int_M W_H d\mu.
\end{eqnarray*}
In particular, if $(M, g(t), \phi(t))$ is a $(0, m)$-super Ricci flow with time dependent metrics and potentials satisfying with the conjugate heat equation \eqref{conjugate}, i.e., 
\begin{eqnarray*}{1\over 2} {\partial g\over \partial t}+Ric_{m, n}(L)\geq 0,\ \ \ {\partial \phi\over \partial t}={1\over 2}{\rm Tr} {\partial g\over \partial t},
\end{eqnarray*}
then $W_H\leq 0$ and 
\begin{eqnarray*}
{d\over dt} W_{m}(u, t)\leq  0.
\end{eqnarray*}
\end{theorem}
{\it Proof}. This follows immediately from the $W$-entropy formula for the time dependent Witten Laplacian on manifolds with super Ricci flows. See \cite{LiLi2015PJM, LiLi2018JFA}. \hfill $\square$

\subsection{LYHP Harnack inequality on super Ricci flows on mm spaces}

Now we state the Li-Yau-Hamilton-Perelman differential Harnack inequality on super Ricci flows on mm spaces. 

\begin{theorem} \label{LYHPmm} \label{LYHPW3}  Let $(X, d(t), g(t), m(t), \phi(t), t\in [0, T])$ be a family of time-dependent $n$-dimensional 
closed RCD metric measures spaces with time dependent metric and potentials satisfying the conjugate heat equation \eqref{conjugate}.  Let \begin{eqnarray*}
\nu_H(t)= [t(2L f-|\nabla f|^2)+f-N]H.
\end{eqnarray*}Then
\begin{eqnarray*}
{d\over dt}(P_{T, t}^*\nu_H(t))=P_{T, t}^*\left(W_H+2\tau [\partial_\tau, L]\log H H\right)\leq 2\tau P_{T, t}^*\left([\partial_\tau, L]\log H H\right),
\end{eqnarray*}
where $P_{T, t}^*$ is the adjoint of the operator $P_{T, t}$ from $L^2((M, g_{T}), \mu)$ to $L^2((M, g_{t}), \mu)$, and
\begin{eqnarray*}
W_H&=&-2t\left[\left\|\nabla^2 f-{g\over 2}\right\|^2_{\rm HS}+2\left({1\over 2}{\partial g\over \partial t}+Ric_{N, n}(L)\right)(\nabla f, \nabla f)\right] H\\
& &\ \ \ \ \ \ -{2t\over N-n}\left(\nabla \phi\cdot\nabla f+{N-n\over 2t}\right)^2H.
\end{eqnarray*}
In particular, if $(X, d, g, m, \phi)$ is a time-independent RCD$(0, n, N)$ space, i.e., 
$Ric_{N, n}(L)\geq 0$, then $[\partial_t, L]=0$ and hence 
\begin{eqnarray*}
{d\over dt}(P_{T, t}\nu_H(t))=P_{T, t}(W_H)\leq 0.
\end{eqnarray*}  
\end{theorem}
{\it Proof}. The proof is similar to the one for Theorem \ref{LYHP}. \hfill $\square$

We can reformulate the W-entropy formula on super Ricci flows on mm spaces as follows.

\begin{theorem} \label{WHmm} Let $(X, d(t), g(t), m(t), \phi(t), t\in [0, T])$ be a family of time-dependent $n$-dimensional 
closed RCD metric measures spaces with time dependent metric and potentials satisfying the conjugate heat equation \eqref{conjugate}. 
Then
\begin{eqnarray*}
W_{N}(u, t)&=&\int_X \nu_H d\mu,
\end{eqnarray*}
and
\begin{eqnarray*}
{d\over dt}W_{N}(u, t)&=&\int_X W_H d\mu.
\end{eqnarray*}
In particular, if $(X,  d(t), g(t), m(t), \phi(t))$ is a $(0, N)$-super Ricci flow on an $n$-dimensional metric 
measure space with time dependent metrics and potentials satisfying the conjugate heat equation \eqref{conjugate}, i.e., 
\begin{eqnarray*}{1\over 2} {\partial g\over \partial t}+Ric_{N, n}(L_t)\geq 0,\ \ \ {\partial \phi\over \partial t}={1\over 2}{\rm Tr} {\partial g\over \partial t},
\end{eqnarray*}
then $W_H\leq 0$ and
\begin{eqnarray*}
{d\over dt} W_{N}(u, t)\leq  0.
\end{eqnarray*}
\end{theorem}

{\it Proof}. The proof is similar to the one for Theorem \ref{WH}. \hfill $\square$

\section{Volume non-local collapsing property and $W$-entropy on mm spaces}

As we have pointed out in the part of Introduction, Perelman \cite{P1}  used the monotonicity of the $W$-entropy on the Ricci flow to prove the non-local collapsing theorem for the Ricci flow and this plays a crucial r\^ole for 
the final resolution of the Poincar\'e conjecture. 

In \cite{N1}, Ni proved that, if $M$ is an $n$-dimensional complete Riemannian manifold with non-negative Ricci curvature, 
then $M$ has maximal volume growth property, namely, 
$$V(B(x,r)) \geq Cr^n, \ \ \forall x\in M, r>0$$ for 
some constant $C > 0$, 
if and only if there exists a constant $A > 0$ such that
$$W(f, \tau)\geq -A, \ \ \forall \tau>0$$
for $u ={e^{-f}\over (4\pi \tau)^{n/2}}$ being the heat kernel of the heat equation $\partial_t u=\Delta u$. In \cite{Li2012}, the author of this paper extended this nice property to the $W$-entropy functional on complete Riemannian manifolds with 
weighted volume measure and with non-negative $m$-dimensional Bakry–Emery Ricci curvature. 

\medskip

The purpose of this section prove the equivalence of the volume non-local collapsing theorem and the lower boundedness of the $W$-entropy on RCD$(0, N)$ spaces.


\medskip

Recall that, by \cite{JLZ},  the fundamental solution to the heat equation $\partial_t u=\Delta u$ satisfies the following two sides estimates on RCD$(-K, N)$ space: Let  $(X, d, \mu)$ be an RCD$
(-K,N)$ space with $K\geq 0$ and $N\in [1, \infty)$. Given
any $\varepsilon>0$, there exist positive constants $C_1(\varepsilon)$ and $C_2(\varepsilon)$, 
depending also on $K$ and $N$, such that for all $x, y\in X$ and $t>0$, it holds
 \begin{equation}\label{JLZB}
 {1\over C_1(\varepsilon) V_x(\sqrt{t}) }\exp\left({-\frac{d^2(x, y)} {(4-\epsilon)t}-C_2(\varepsilon) t } \right)\leq p_t(x, y)\leq 
 { C_1(\varepsilon)\over V_x(\sqrt{t}) }\exp\left({-\frac{d^2(x, y)} {(4+\epsilon)t}+C_2(\varepsilon) t } \right),
 \end{equation}
 where $V_x(\sqrt{t})=\mu(B(x, \sqrt{t})$ is the volume of the ball $B(x, \sqrt{t})=\{y\in X: d(x, y)\leq \sqrt{t}\}$. 

\begin{theorem}  Let $(X, d, \mu)$ be an RCD$(0, N)$ space. Then $(X, d, \mu)$ has the volume non-collapsing property, namely, for some constant $C> 0$ and $r_0>0$,
\begin{eqnarray}\label{noncollapsing}
\mu(B(x,r))\geq Cr^N, \ \ \forall r\in (0, r_0], \ \forall x\in X, 
\end{eqnarray}
if and only if there exists a constant $A > 0$ such that 
\begin{eqnarray}\label{WgeqA}
W_N(f, \tau)\geq -A,\ \ \forall \tau\in (0, r_0^2] 
\end{eqnarray}
for $u ={e^{-f}\over (4\pi \tau)^{N/2}}$ being the heat kernel of the heat equation $\partial_t u=Lu$. When $r_0=+\infty$, the global maximal volume growth condition
\begin{eqnarray}\label{maxvolumegrowth}
\mu(B(x,r))\geq Cr^N, \ \ \forall r\geq 0, \ \forall x\in X 
\end{eqnarray}
 is equivalent to the lower boundedness of the $W$-entropy $W(f, \tau)$ for all $\tau\in (0, \infty)$, i.e., 
  \begin{eqnarray}\label{globalWgeqA}
W_N(f, \tau)\geq -A,\ \ \forall \tau\in (0, \infty). 
\end{eqnarray}
\end{theorem}
{\it Proof}. The proof is very close the ones given in Ni \cite{N1} for complete Riemannian mnaifolds with non-negative Ricci curvature and in our previous paper  \cite{Li2012} for weighted complete 
Riemannian manifolds with CD$(0, m)$-condition. Due to the importance of this result and for the completeness of the paper, we allow us to reproduce it as follows.  Suppose that \eqref{noncollapsing} holds. Let $v=\sqrt{u}$. Then we can rewrite $W_N(f, \tau)$ as 
\begin{eqnarray}
\label{Wref}
W_N(f,\tau)=4\tau\int_X  |\nabla v|^2d\mu-\int_X v^2\log v^2d\mu-N+{N\over 2}\log(4\pi \tau).
\end{eqnarray}
On RCD$(0, N)$ space, we have the Li-Yau heat kernel upper bound estimate  (see \cite{ZZ, Jiang, JLZ})

\begin{eqnarray}\label{volestimate}
v^2\leq { C(N) \over \mu(B(x, \sqrt{\tau})}\leq {C(N)C\over 
\tau^{N/2}}, \ \ \  \forall \tau\in (0, r_0^2], 
\end{eqnarray}
from which we get

\begin{eqnarray}\label{Westimate}
   W_N(f, \tau)\geq -\log (C(N)C)-N-{N\over 2}\log(4\pi), \ \ \  \forall \tau\in (0, r_0^2], 
\end{eqnarray}
This proves that $W_N(f, \tau)$ is bounded from below, i.e., \eqref{WgeqA}.

Conversely, if $W_N(f, \tau)\geq -A$ for some constant $A\geq 0$ and for all $\tau\in (0, r_0^2]$,
we want to prove \eqref{noncollapsing} holds for some constant $C = C(N, A)$ and for all
 $r\in (0, r_0]$. To this end, we use the lower bound estimate of the heat kernel as well as the 
 Li-Yau Harnack inequality on RCD spaces. In fact, 
on RCD$(0, N)$ space, the  Li-Yau Harnack differential inequality holds (\cite{Jiang, JLZ, ZZ})

\begin{eqnarray}\label{Li-Yau}
{|\nabla u|^2\over  u^2}-{\partial_\tau u\over u}\leq {N\over 2\tau}, \ \ \forall \tau>0.
\end{eqnarray}
 Thus, for all $\tau > 0$, we have

\begin{eqnarray*}
4\tau \int_X |\nabla v|^2 d\mu=\tau \int_X {|\nabla u|^2\over  u} d\mu\leq \tau \int_X \left({\partial_\tau u\over u}+ {N\over 2\tau}\right) ud\mu={N\over 2}.
\end{eqnarray*}
Moreover, using the lower bound estimate of the heat kernel \eqref{JLZB}, we have
\begin{eqnarray*}
-\int_X v^2\log v^2 d\mu &\leq& -\int_X \log\left({C_3(N)\over \mu(B(x, \sqrt{\tau})} e^{-{d^2(x, y)
\over 3\tau}} \right)ud\mu\\
&\leq& C_3(N)+\log \mu(B(x, \sqrt{\tau})+{1\over 3\tau}\int_X d^2(x, y)u(x, y, \tau)d\mu(y).
\end{eqnarray*}
  Based on the Li-Yau upper bound estimate \eqref{JLZB},  
  we can prove that

\begin{eqnarray*}
\int_X d^2(x, y)u(x, y, \tau)d\mu(y)\leq C_4(N).
\end{eqnarray*}
 Therefore
 
\begin{eqnarray*}
-\int_X v^2\log v^2 d\mu &\leq& C_5(N)+\log \mu(B(x, \sqrt{\tau})).
\end{eqnarray*}
Substituting the above estimates and making use of the assumption
 $W_N( f , \tau)\geq -A$ for all $\tau\in (0, r_0^2]$ into (31), we have

\begin{eqnarray}
\log \mu(B(x, \sqrt{\tau}))\geq {N\over 2}\log(4\pi \tau)-C_6(N)-A \ \ \ \forall \tau\in (0, r_0^2].
\end{eqnarray}
Equivaleently

\begin{eqnarray}
\mu(B(x, r))\geq (4\pi)^{N\over 2}e^{-(A+C_6(N))}r^{N},
 \ \ \ \forall r\in (0, r_0].
\end{eqnarray}
Here $C_i (N)$, $i=1, \ldots, 6$,  denote positive constants depending only on $N$.
The proof of theorem is completed. \hfill $\square$

\begin{remark}
Indeed, as pointed out by Ni \cite{N1}, the similar result as above was claimed in Perelman \cite{P1} for the Ricci flow ancient solutions. The proof for Proposition 4.2  in Ni \cite{N1} is easier than the nonlinear case considered in \cite{P1}. In fact, Proposition 4.2 in Ni \cite{N1} can be used in the proof of Theorem 10.1 of \cite{P1}.
\end{remark}

Indeed, we can also prove the following result which extends Ni's Corollary 4.3 in \cite{N1}. 

\begin{corollary} Let $u={e^{-f}\over (4\pi t)^{N/2}}$ be the fundamental solution to the heat equation $\partial_t u=Lu$ on an 
RCD$(0, N)$ space. Suppose that the maximal volume growth condition \eqref{maxvolumegrowth} 
holds, equivalently,  the global lower boundedness condition \eqref{globalWgeqA} of 
the $W$-entropy holds. Then $W_\infty:=\lim\limits_{t\rightarrow \infty} W_N(f, t)$ and $\kappa:=
\lim\limits_{r\rightarrow \infty}{\mu(B(x, r))\over \omega_N r^N}$ exist, where $\omega_N$  is the volume
of the unit ball in $\mathbb{R}^N$. Moreover, we have
$$W_\infty=\log \kappa$$
 \end{corollary}
 {\it Proof}. The proof is similar to the one of Corollary 4.3 of \cite{N1} given in \cite{N2}. See also S. Li-Li \cite{LiLi2024TMJ}. Let $H_N(u, t)=H(u(t))-{N\over 2}\log(4\pi et)$ be the Nash entropy as introduced in Ni \cite{N1, N2} and Li \cite{Li2012}, and let 
 $F_N(u, t)={dH_N(u, t)\over dt}$. Then  $W_N(f, t) = tF_N(u, t) + H_N(u, t)$. Similarly to the case of complete Riemannian manifolds with CD$(0, m)$-condition as we studied in \cite{Li2012, LiLi2024TMJ}, the Li-Yau Harnack inequality on RCD$(0, N)$ space \cite{Jiang, ZZ} implies $F_N(u, t)={dH_N(u, t)\over dt}\leq 0$. Hence 
 $\lim\limits_{N\rightarrow \infty}H_N(u, t)$ exists.  By \cite{N1, N2, HLi, LiLi2024TMJ}, 
 under the assumption \eqref{maxvolumegrowth}  or \eqref{globalWgeqA}, 
  $\lim\limits_{t\rightarrow \infty} H_N(u, t)=\log \kappa$. Hence 
  $|H_N(u, 2t)-H_N(u, t)|\leq \varepsilon$ for $t>>1$. This implies that there exists $t_i$ such that 
 $t_iF_N(u, t_i) \rightarrow 0$
as $t_i \rightarrow \infty$. The monotonicity of $W_N(f, t) = tF_N(u, t) + H_N(u, t)$ implies that 
 $\lim\limits_{t\rightarrow \infty}W_N(f, t) =\lim\limits_{t\rightarrow \infty} H_N(u, t)=\log \kappa$.  This completes the proof. \hfill $\square$
 
\section{Logarithmic Sobolev inequality and $W$-entropy on mm spaces}

By \cite{P1}, it has been well-known that the $W$-entropy is closely related to a family of 
Log-Sobolev inequalities on Riemannian manifolds and Ricci flow. The following result is an extension  of Theorem 6.2 in \cite{Li2012} which was proved in the setting of compact Riemannian manifolds. 

\begin{theorem}\label{Soblog} Let $(X, d, \mu)$ be an RCD space. Assume that 
the $L^2$-Sobolev 
inequality holds: there exists a constant $C_{\rm Sob}>0$ such that for all $f\in W^{1, 2}(X, \mu)$, 
$$\|f\|_{2N\over N-2}^2 \leq C_{\rm Sob} (\|\nabla f\|_2^2+\|f\|_2^2). $$
Then for any $\tau>0$ there exists a constant $\mu(\tau)>-\infty$ such that the following Log-Sobolev inequality holds: for all $f\in W^{1, 2}(X, \mu)$ with $\int_X f^2 d\mu=1$, 
\begin{eqnarray}\label{LSImu}
\int_X f^2\log f^2d\mu\leq 4\tau \|\nabla f\|_2^2-\left(1+{1\over 2}\log (4\pi\tau)\right)
-\mu(\tau).
\end{eqnarray}
Indeed, $\mu(\tau)$ is the optimal constant in the above Log-Sobolev inequality 
$$
\mu(\tau):=\inf\left\{\int_X \left[4\tau |\nabla u|^2-u^2\log u^2-Nu^2\right] 
{d\mu\over (4\pi \tau)^{N/2}}:\ \  \int_X (4\pi t)^{-N/2} u^2d\mu=1 \right\}>-\infty.$$
\end{theorem}

{\it Proof}. By Davies \cite{Davis}, it is well-known that the $L^2$-Sobolev inequality implies a family of Log-Sobolev inequalities: for any $\varepsilon>0$, 
there exists a constant $\beta(\varepsilon)>0$ such that 
$$\int_X f^2\log f^2d \mu\leq \varepsilon \|\nabla f\|_2^2+\beta(\varepsilon) \|f\|_2^2
+\|f\|_2\log \|f\|_2, \ \ \ \forall f\in W^{1, 2}(X, \mu),$$
where for some constant $C > 0$, it holds
$$ \beta(\varepsilon)\leq C-N\log \varepsilon.$$
Taking $\varepsilon=4\tau$ and defining 
$$
-\mu(\tau):=\beta(4\tau)+N\left(1+{1\over 2}\log (4\pi\tau)\right),$$
then $\mu(\tau)\geq -\left(C+N+{N\over 2}\log(4\pi \tau)\right)>-\infty$ and the  Log-Sobolev inequality 
\eqref{LSImu} holds.  This finishes the proof of theorem. \hfill $\square$
 
 Concerning the $L^2$-Sobolev inequality as used in Theorem \ref{Soblog}, we would like to recall that, in our previous paper  \cite{LiJGA2009}, we proved the following $L^p$-Sobolev inequality on complete Riemannian mnaifolds with CD$(K, m)$-condition.
 
 \begin{theorem} [See Theorem 7.2 in \cite{LiJGA2009}]  Let $M$ be a 
 complete Riemannian manifold on which the $m$-dimensional Bakry-Emery 
 Ricci curvature is uniformly bounded from below by a negative constant $K$, 
 i.e., $Ric_{m,n}(L) \geq K$, where $K$ is a negative constant. 
 Suppose that there exist two constants $\alpha\in (2, m]$ and $C_\alpha> 0$ 
 such that

\begin{eqnarray}
\mu((B(x,r))\geq C_\alpha r^\alpha, \ \ \forall x\in M, r>0. \label{Li44}
\end{eqnarray}
Then, for all $p\in (1, \alpha)$, and for $q=q(p, \alpha)$ given by
$${1\over q} ={1\over p}-{1\over \alpha}$$
we have 
\begin{eqnarray}
\|f \|_q \leq C_{m,p,\alpha}(\|\nabla f\|_p + \|f \|_p),  \ \ \ \forall f\in C_0^\infty(M). \label{Li45}
\end{eqnarray}
In particular, if $Ric_{m,n}(L)\geq 0$, and \eqref{Li44} holds, then for all 
$p\in  (1, \alpha)$, and with $q(p, \alpha)$
as given above, we have

\begin{eqnarray}
\|f \|_q \leq C_{m,p,\alpha}(\|\nabla f\|_p,  \ \ \ \forall f\in C_0^\infty(M). \label{Li46}
\end{eqnarray}
\end{theorem}

The proof of the above theorem only relies on the upper bound heat kernel estimate and Varopoulos' 
Littlewood-Paley 
theory of the ultracontractive semigroup  \cite{Var}. It can be easily adapted to RCD spaces. Thus, 
by the same argument as in the proof of Theorem 7.2 in \cite{LiJGA2009}, we can prove the following 

\begin{theorem}  Let $(X, d, \mu)$ be an
RCD$(K, N)$ space, $N\geq 2$ and $K\leq 0$ are two constants. 
 Suppose that there exist two constants $\alpha\in (2, N]$ and $C_\alpha> 0$ 
 such that

\begin{eqnarray}
\mu((B(x,r))\geq C_\alpha r^\alpha, \ \ \forall x\in X, r>0. \label{Li44b}
\end{eqnarray}
Then, for all $p\in (1, \alpha)$, and for $q=q(p, \alpha)$ given by
$${1\over q} ={1\over p}-{1\over \alpha}$$
we have 
\begin{eqnarray}
\|f \|_q \leq C_{m,p,\alpha}(\|\nabla f\|_p + \|f \|_p),  \ \ \ \forall f\in W^{1, p}(X, \mu). \label{Li45b}
\end{eqnarray}
In particular, on RCD$(0, N)$ space with the volume growth condition 
\eqref{Li44b}, the Euclidean Sobolev inequality holds, i.e.,  for all 
$p\in  (1, \alpha)$, and with $q(p, \alpha)$
as given above, we have
\begin{eqnarray}
\|f \|_q \leq C_{m,p,\alpha}(\|\nabla f\|_p,  \ \ \ \forall f\in W^{1, p}(X, \mu). \label{Li46b}
\end{eqnarray}
\end{theorem}

 The following result extends the known results in the case of Riemannian manifolds with CD$(K, m)$-condition or  smooth $(K, m)$-super Ricci flows, see \cite{Li2012, LiLi2015PJM}. 
  
\begin{theorem}  Let $(X, d_t, g_t, m_t, \phi_t, t\in [0, T])$ be a closed $(K, n, N)$-super Ricci flow on mm space. 
Then the extremal function $u=e^{-v/2}\in W^{1, 2}(X, \mu)$ which achieves the optimal Log-Sobolev constant $\mu_K(t)$ defined by
\begin{equation}\label{LSI}
\mu_K(t):=\inf\left\{W_{N, K}(u,t):\int_X\frac{e^{-v}}{(4\pi t)^{N/2}}d\mu=1\right\},
\end{equation}
satisfies the Euler-Lagrange equation
\begin{equation}\label{ELLSI}
-4tLu-2u\log u-N\left(1-\frac {K}{2t}\right)^2u=\mu_K(t)u.
\end{equation}
Moreover, if $(X, d_t, g_t, m_t, \phi_t, t\in [0, T])$ is a $(K, m)$-super Ricci flow with the conjugate equation \eqref{conjugate}, then $\mu_K(t)$ is decreasing in $t\in [0, T]$.
\end{theorem}
{\it Proof}. The proof is similar to the one given by Perelman \cite{P1}, see also \cite{Li2012, LiLi2015PJM}.  \hfill $\square$

\medskip
\begin{remark}{In the case of Riemannian manifolds or smooth Ricci or super Ricci flows, the Schauder regularity theory of nonlinear elliptic PDEs leads us to derive $u\in C^{2,\alpha}(M)$ for $\alpha\in(0,1)$. Then, an argument due to Rothaus \cite{Roth}
allows them to prove that $u$ is strictly positive and smooth. This yields that $v = -2\log u$ is also
smooth. It would be interesting to see what happens on 
RCD$(K,N)$ spaces and closed $(K, n, N)$-super Ricci flows. This suggest us to study the Schauder and De Giorgi-Moser-Nash regularity theory of nonlinear elliptic PDEs on RCD spaces and super Ricci flows on metric measure spaces.
}
\end{remark}

\vspace{5mm}

Xiang-Dong Li, State Key Laboratory of Mathematical Sciences, Academy of Mathematics and Systems Science, Chinese Academy of Sciences, No. 55, Zhongguancun East Road, Beijing, 100190, China,  
and 
 School of Mathematics, University of Chinese Academy of Sciences, Beijing, 100049, China. Email: xdli@amt.ac.cn

\end{document}